\documentclass{mcom-l}
\usepackage{booktabs}
\usepackage{textcomp}
\usepackage{graphicx}
\usepackage{tikz}
\usepackage{amssymb,amsfonts,amsmath,amsthm}
\usepackage{url}

\renewcommand{\O}{\mathcal{O}}
\newcommand{\Z}{\mathbb{Z}}
\newcommand{\Q}{\mathbb{Q}}
\newcommand{\Fp}{\mathbb{F}_p}

\newcommand{\Fq}{\mathbb{F}_q}
\newcommand{\cl}{\operatorname{cl}}
\newcommand{\clD}{\cl(D)}
\newcommand{\Ell}{\operatorname{Ell}}
\newcommand{\Ellt}{\Ell_t}
\newcommand{\EllD}{\Ell_\O}
\newcommand{\End}{\operatorname{End}}
\newcommand{\Glt}{\Gamma_{\ell,t}}
\newcommand{\M}{\mathsf{M}}
\newcommand{\tr}{\operatorname{tr}}

\newcommand{\PD}{\mathcal{P}_D}
\newcommand{\DPF}{\mathcal{D}_{\scriptscriptstyle\rm PF}}
\newcommand{\pmax}{p_{\scriptscriptstyle\rm{M}}}
\newcommand{\vmax}{v_{\scriptscriptstyle\rm{M}}}
\newcommand{\lmax}{\ell_{\scriptscriptstyle\rm{M}}}

\newcommand{\inkron}[2]{(\frac{#1}{#2})}
\def\Gal{\operatorname{Gal}}
\renewcommand{\vec}[1]{{\boldsymbol{#1}}}

\newcommand{\algstart}[2]{\smallskip\noindent{\bf Algorithm} #1. \emph{#2}\begin{enumerate}}
\newcommand{\algend}{\end{enumerate}\vspace{4pt}}
\newcommand{\algitem}{\vspace{1pt}\item}
\newcommand{\ord}{\operatorname{ord}}

\newcommand{\llog}{\operatorname{llog}}
\newcommand{\lllog}{\operatorname{lllog}}
\newcommand{\Li}{\operatorname{Li}}
\newcommand{\slt}{\medspace<\medspace}
\newcommand{\sle}{\medspace\le\medspace}
\newcommand{\seq}{\medspace=\medspace}
\newcommand{\sge}{\medspace\ge\medspace}
\newcommand{\sgt}{\medspace>\medspace}

\newtheorem*{thm1}{Theorem~1}
\newtheorem{proposition}{Proposition}
\newtheorem{corollary}{Corollary}
\newtheorem{lemma}{Lemma}

\theoremstyle{definition}
\newtheorem{definition}{Definition}

\renewcommand\labelenumi{\theenumi.}

\begin{document}

\title[Computing Hilbert class polynomials with the CRT]{Computing Hilbert class polynomials\\with the Chinese Remainder Theorem}

\author{Andrew V. Sutherland}
\address{Massachusetts Institute of Technology, Cambridge, Massachusetts 02139}
\curraddr{}
\email{drew@math.mit.edu}
\copyrightinfo{2009}{by the author}
\thanks{}

\subjclass[2000]{Primary 11Y16; Secondary 11G15, 11G20, 14H52}

\date{}

\dedicatory{}

\begin{abstract}
We present a space-efficient algorithm to compute the Hilbert class polynomial $H_D(X)$ modulo a positive integer $P$, based on an explicit form of the Chinese Remainder Theorem.  Under the Generalized Riemann Hypothesis, the algorithm uses $O(|D|^{1/2+\epsilon}\log{P})$ space and has an expected running time of $O(|D|^{1+\epsilon})$.  We describe practical optimizations that allow us to handle larger discriminants than other methods, with $|D|$ as large as $10^{13}$ and $h(D)$ up to $10^{6}$.  We apply these results to construct pairing-friendly elliptic curves of prime order, using the CM method.
\end{abstract}

\maketitle

\section{Introduction}
Elliptic curves with a prescribed number of points have many applications, including elliptic curve primality proving \cite{Atkin:ECPP} and pairing-based cryptography \cite{Freeman:PairingTaxonomy}.  The number of points on an elliptic curve $E/\Fq$ is of the form $N=q+1-t$, where $|t|\le 2\sqrt{q}$.  For an ordinary elliptic curve, we additionally require $t\not\equiv 0\bmod p$, where $p$ is the characteristic of $\Fq$.  We may construct such a curve via the $\emph{CM method}$.

To illustrate, let us suppose $D<-4$ is a quadratic discriminant satisfying
\begin{equation}\label{equation:TotallySplitPrime}
4q=t^2-v^2D,
\end{equation}
for some integer $v$, and let $\O$ denote the order of discriminant $D$.
The $j$-invariant of the elliptic curve $\mathbb{C}/\O$ is an algebraic integer, and
its minimal polynomial $H_D(X)$ is the {\it Hilbert class polynomial} for the discriminant $D$.
This polynomial splits completely in $\Fq$, and its roots are the $j$-invariants of elliptic curves with endomorphism ring isomorphic to $\O$.  To construct such a curve, we reduce $H_D\bmod p$, compute a root in $\Fq$, and define an elliptic curve $E/\Fq$ with this $j$-invariant.  Either $E$ or its quadratic twist has $N$ points, and we may easily determine which.  For more details on constructing elliptic curves with the CM method, see \cite{Atkin:ECPP,Broker:EfficientCMConstructions,Lay:CM}.

The most difficult step in this process is obtaining $H_D$, an integer polynomial of degree $h(D)$ (the class number) and total size $O(|D|^{1+\epsilon})$ bits.  There are several algorithms that, under reasonable heuristic assumptions, can compute $H_D$ in quasi-linear time \cite{Belding:HilbertClassPolynomial,Broker:pAdicClassPolynomial,Couveignes:ClassPolynomial,Enge:FloatingPoint}, but its size severely restricts the feasible range of~$D$.  The bound $|D|<10^{10}$ is commonly cited as a practical upper limit for the CM method \cite{Freeman:PairingTaxonomy,Kachisa:PairingFriendly,Karabina:MNTCurves,Vercauteren:IDCrypto}, and this already assumes the use of alternative class polynomials that are smaller (and less general) than $H_D$.  As noted in \cite{Enge:FloatingPoint}, \emph{space} is the limiting factor in these computations, not running time.  But the CM method only uses $H_D\bmod p$, which is typically much smaller than $H_D$.

We present here an algorithm to compute $H_D\bmod P$, for any positive integer~$P$, using $O(|D|^{1/2+\epsilon}\log{P})$ space.  This includes the case where $P$ is larger than the coefficients of $H_D$ (for which we have accurate bounds), hence it may be used to determine $H_D$ over $\Z$.  Our algorithm is based on the CRT approach \cite{Agashe:CRTClassPolynomial,Belding:HilbertClassPolynomial,Chao:CRTCMmethod}, which computes the coefficients of $H_D$ modulo many ``small" primes $p$ and then applies the Chinese Remainder Theorem (CRT).  As in \cite{Agashe:CRTClassPolynomial}, we use the explicit CRT \cite[Thm.~3.1]{Bernstein:ModularExponentiation} to obtain $H_D\bmod P$, and we
modify the algorithm in \cite{Belding:HilbertClassPolynomial} to compute $H_D\bmod p$ more efficiently.
Implementing the CRT computation as an online algorithm reduces the space required.  We obtain a probabilistic algorithm to compute $H_D\bmod P$ whose output is always correct (a \emph{Las Vegas} algorithm).

Under the Generalized Riemann Hypothesis (GRH), its expected running time is $O(|D|^{1+\epsilon})$.  More precisely, we prove the following theorem.

\begin{thm1}
Under the GRH, {\rm Algorithm}~$2$ computes $H_D\bmod P$ in expected time $O\bigl(|D|\log^{5}|D|(\log\log|D|)^4\bigr),$ using $O\bigl(|D|^{1/2}(\log|D|+\log{P})\log\log|D|\bigr)$ space.
\end{thm1}
\noindent
In addition to the new space bound, this improves the best rigorously proven time bound for computing $H_D$ under the GRH \cite[Thm.~1]{Belding:HilbertClassPolynomial}, by a factor of $\log^2|D|$.
Heuristically, the time complexity is $O(|D|\log^{3+\epsilon}|D|)$.
We also describe practical improvements that make the algorithm substantially faster than alternative methods when $|D|$ is large, and provide
computational results for $|D|$ up to $10^{13}$ and $h(D)$ up to $10^6$.  In our largest examples the total size of $H_D$ is many terabytes, but less than 200 megabytes are used to compute $H_D$ modulo a 256-bit prime.

\section{Overview}\label{section:Overview}
Let $\O$ be a quadratic order with discriminant $D<-4$.  With the CRT approach, we must compute $H_D\bmod p$ for many primes $p$.  We shall use primes in the set
\begin{equation}\label{equation:PD}
\PD = \{p>3\medspace {\rm prime}: 4p=t^2-v^2D \medspace\text{for some} \medspace t, v\in \Z_{>0}\}.
\end{equation}
These primes split completely in the ring class field $K_\O$ of $\O$, split into principal ideals in $\Q[\sqrt{D}]$, and are norms of elements in $\O$, see \cite[Prop.~2.3, Thm.~3.2]{Atkin:ECPP}.
For each $p\in\PD$, the positive integers $t=t(p)$ and $v=v(p)$ are uniquely determined.

We first describe how to compute $H_D\bmod p$ for a prime $p\in \PD$, and then explain how to obtain $H_D\bmod P$ for an arbitrary positive integer $P$.  Let us begin by recalling a few pertinent facts from the theory of complex multiplication.

For any field $F$, we define the set
\begin{equation}\label{equation:EllD}
\EllD(F) = \{j(E/F):\End(E)\cong\O\},
\end{equation}
the $j$-invariants of elliptic curves defined over $F$ whose endomorphism rings are isomorphic to $\O$.
There are two possibilities for the isomorphism in (\ref{equation:EllD}), but as in \cite{Belding:HilbertClassPolynomial} we make a canonical choice and henceforth identify $\End(E)$ with $\O$.  For $j(E)\in\EllD(F)$ and an invertible ideal $\mathfrak{a}$ in $\O$, let $E[\mathfrak{a}]$ denote the group of $\mathfrak{a}$-torsion points, those points annihilated by every $z\in\mathfrak{a}\subseteq\O\cong\End(E)$.  We then define
\begin{equation*}
j(E)^\mathfrak{a}=j(E/E[\mathfrak{a}]).
\end{equation*}
The map $j(E)\mapsto j(E)^\mathfrak{a}$ corresponds to an isogeny with kernel $E[\mathfrak{a}]$ and degree equal to the norm of $\mathfrak{a}$.  This yields a group action of the ideal group of $\O$ on the set $\EllD(K_\O)$, and this action factors through the class group $\cl(\O)=\clD$.

For a prime $p\in\PD$, a bijection between $\EllD(\Fp)$ and $\EllD(K_\O)$ arises from the Deuring lifting theorem, see \cite[Thms.~13.12-14]{Lang:EllipticFunctions}. The following proposition then follows from the theory of complex multiplication.
\begin{proposition}\label{prop:CM}
For each prime $p\in\PD$:
\begin{enumerate}
\item
$H_D(X)$ splits completely over $\Fp$.  It has $h(D)$ roots, which form $\EllD(\Fp)$.
\item
The map $j(E)\mapsto j(E)^{\mathfrak{a}}$ defines a free transitive action of $\clD$ on $\EllD(\Fp)$.
\end{enumerate}
\end{proposition}
\noindent
For further background, we recommend the expositions in \cite{Cox:ComplexMultiplication} and \cite{Serre:ComplexMultiplication}, and also the material in \cite[Ch.~10]{Lang:EllipticFunctions} and \cite[Ch.~II]{Silverman:EllipticCurves2}.
\medskip

Let $p$ be a prime in $\PD$.
Our plan is to compute $H_D\bmod p$ by determining its roots and forming the product of the corresponding linear factors.
By Proposition~\ref{prop:CM}, we can obtain the roots by enumerating the set $\EllD(\Fp)$ via the action of $\clD$.
All that is required is an element of $\EllD(\Fp)$ to serve as a starting point.  Thus we seek an elliptic curve $E/\Fp$ with $\End(E)\cong\O$.  Now it may be that very few elliptic curves $E/\Fp$ have this endomorphism ring.
Our task is made easier if we first look for an elliptic curve that at least has the desired Frobenius endomorphism, even if its endomorphism ring might not be isomorphic to $\O$.

For $j(E)\in\EllD(\Fp)$, the Frobenius endomorphism $\pi_E\in\End(E)\cong\O$ corresponds to an element of $\O$ with norm $p$ and trace $t$.  Let us consider the set
\begin{equation}\label{equation:ellt}
\Ellt(\Fp) = \{j(E/\Fp): \tr(\pi_E)=t\},
\end{equation}
the $j$-invariants of all elliptic curves $E/\Fp$ with trace $t$.
We may regard $j\in\Ellt(\Fp)$ as identifying a particular elliptic curve $E/\Fp$ satisfying $j(E)=j$ and $\tr(\pi_E)=t$, since such an $E$ is determined up to isomorphism \cite[Prop.~14.19]{Cox:ComplexMultiplication}.
We have $\EllD(\Fp)\subseteq\Ellt(\Fp)$, and note that $\Ellt(\Fp)=\Ell_{-t}(\Fp)$.

Recall that elliptic curves $E/\Fp$ and $E'/\Fp$ are isogenous over $\Fp$ if and only if $\tr(\pi_E)=\tr(\pi_E')$, see \cite[Thm.~13.8.4]{Husemoller:EllipticCurves}.  Given $j(E)\in\Ellt(\Fp)$, we can efficiently obtain an isogenous $j(E')\in\EllD(\Fp)$, provided $v(p)$ has no large prime factors.

This yields Algorithm~1.  Its structure matches \cite[Alg.~2]{Belding:HilbertClassPolynomial}, but we significantly modify the implementation of Steps 1, 2, and 3.

\algstart{{\bf 1}}{Given $p\in\PD$, compute $H_D\bmod p$ as follows:}
\algitem
Search for a curve $E$ with $j(E)\in\Ellt(\Fp)$ (Algorithm~1.1).
\algitem
Find an isogenous $E'$ with $j(E')\in\EllD(\Fp)$ (Algorithm~1.2).
\algitem
Enumerate $\EllD(\Fp)$ from $j(E')$ via the action of $\clD$ (Algorithm~1.3).
\algitem
Compute $H_D\bmod p$ as $H_D(X) = \prod_{j\in\EllD(\Fp)}(X-j)$.
\algend

Algorithm~1.1 searches for $j(E)\in\Ellt(\Fp)$ by sampling random curves and testing whether they have trace $t$ (or $-t$).  To accelerate this process, we sample a family of curves whose orders are divisible by $m$, for some suitable $m|(p+1\pm t)$.  We select $p\in\PD$ to ensure that such an $m$ exists, and also to maximize the size of $\Ellt(\Fp)$ relative to $\Fp$ (with substantial benefit).

To compute the isogenies required by Algorithms 1.2 and 1.3 we use the classical modular polynomial $\Phi_N\in\Z[X,Y]$, which parametrizes elliptic curves connected by a cyclic isogeny of degree $N$.  For a prime $\ell\ne p$ and an elliptic curve $E/\Fp$, the roots of $\Phi_\ell(X,j(E))$ over $\Fp$ are the $j$-invariants of all curves $E'/\Fp$ connected to~$E$ via an isogeny of degree $\ell$ (an $\ell$-isogeny) \cite[Thm.~12.19]{Washington:EllipticCurves}.  This gives us a computationally explicit way to define the graph of $\ell$-isogenies on the set $\Ellt(\Fp)$.

As shown by Kohel \cite{Kohel:thesis}, the connected components of this graph all have a particular shape, aptly described in \cite{Fouquet:IsogenyVolcanoes} as a \emph{volcano} (see Figure~1 in Section~\ref{section:IsogenyVolcanoes}).  The curves in an isogeny volcano are naturally partitioned into one or more levels, according to their endomorphism rings, with the curves at the top level forming a cycle.  Given an element of $\Ellt(\Fp)$, Algorithm~1.2 finds an element of $\EllD(\Fp)$ by climbing a series of isogeny volcanoes.  Given an element of $\EllD(\Fp)$, Algorithm~1.3 enumerates the entire set by walking along $\ell$-isogeny cycles for various values of $\ell$.
\smallskip

We now suppose we have computed $H_D$ modulo primes $p_1,\ldots,p_n$ and consider how to compute $H_D\bmod P$ for an arbitrary positive integer $P$,
using the Chinese Remainder Theorem.  In order to do so, we need an explicit bound $B$ on the largest coefficient of $H_D$ (in absolute value).  Lemma~\ref{lemma:Bbound} of Appendix~1 provides such a $B$, and it satisfies $\log B = O(|D|^{1/2+\epsilon})$.

Let $M=\prod p_i$, $M_i=M/p_i$ and $a_i \equiv M_i^{-1} \bmod p_i$.  Suppose $c\in \Z$ is a coefficient of $H_D$.  We know the values $c_i\equiv c\bmod p_i$ and wish to compute $c\bmod P$ for some positive integer $P$.
We have
\begin{equation}\label{equation:CRT}
c \equiv \sum c_ia_iM_i \bmod M,
\end{equation}
and if $M>2B$ we can uniquely determine $c$.  This is the usual CRT approach.

Alternatively, if $M$ is slightly larger, say $M>4B$, we may apply the explicit CRT (mod $P$) \cite[Thm.~3.1]{Bernstein:ModularExponentiation}, and compute $c\bmod P$ directly via
\begin{equation}\label{equation:explicitCRT}
c \equiv \sum c_ia_iM_i-rM \bmod P.
\end{equation}
Here $r$ is the nearest integer to $\sum c_ia_i/p_i$.  When computing $r$ it suffices to approximate each rational number $c_ia_i/p_i$ to within $1/(4n)$.

As noted in \cite{Enge:FloatingPoint}, even when $P$ is small one still has to compute $H_D\bmod p_i$ for enough primes to determine $H_D$ over $\Z$, so the work required is essentially the same.  The total size of the $c_i$ over all the coefficients is necessarily as big as $H_D$.

However, instead of applying the explicit CRT at the end of the computation, we update the sums $\sum c_ia_iM_i\bmod P$ and $\sum c_ia_i/p_i$ as each $c_i$ is computed and immediately discard $c_i$.  This \emph{online} approach reduces the space required.

We now give the complete algorithm to compute $H_D\bmod P$.  When $P$ is large
we alter the CRT approach slightly as described in Section 7.  This allows us to efficiently treat all $P$, including $P=M$, which is used to compute $H_D$ over $\Z$.

\algstart{{\bf 2}}{Compute $H_D\bmod P$ as follows:}
\algitem
Select primes $p_1,\ldots,p_n\in\PD$ with $\prod p_i > 4B$ (Algorithm~2.1).
\algitem
Compute suitable presentations of $\clD$ (Algorithm~2.2).
\algitem
Perform CRT precomputation (Algorithm~2.3).
\algitem
For each $p_i$:
\begin{enumerate}
\item
Compute the coefficients of $H_D\bmod p_i$ (Algorithm~1).
\item
Update CRT sums for each coefficient of $H_D$ (Algorithm~2.4).
\end{enumerate}
\algitem
Recover the coefficients of $H_D\bmod P$ (Algorithm~2.5).
\algend

The presentations computed by Algorithm~2.2 are used by Algorithm~1.3 to realize the action of the class group.
The optimal presentation may vary with~$p_i$ (more precisely, $v(p_i)$), but often the same presentation is used for every~$p_i$.
Each presentation specifies a sequence of primes $\ell_1,\ldots,\ell_k$ corresponding to a sequence $\alpha_1,\ldots,\alpha_k$ of generators for $\clD$ in which each $\alpha_i$ contains an ideal of norm $\ell_i$.  There is an associated sequence of integers $r_1,\ldots,r_k$ with the property that every $\beta\in\cl(D)$ can be expressed uniquely in the form
\begin{equation*}
\beta = \alpha_1^{x_1}\cdots\alpha_k^{x_k},
\end{equation*}
with $0\le x_i < r_i$.  Algorithm~1.3 uses isogenies of degrees $\ell_1,\ldots,\ell_k$ to enumerate $\EllD(\Fp)$.  Given the large size of $\Phi_\ell(X,Y)$, roughly $O(\ell^3\log \ell)$ bits \cite{CohenPaula:ModularPolynomials}, it is critical that the $\ell_i$ are as small as possible.  We achieve this by computing an optimal \emph{polycyclic presentation} for $\clD$, derived from a sequence of generators for $\cl(D)$.  Under the Extended Reimann Hypothesis (ERH) we have $\ell_i\le 6\log^2|D|$, by \cite{Bach:ERHBounds}.  This approach corrects an error in \cite{Belding:HilbertClassPolynomial} which relies on a \emph{basis} for $\clD$ and fails to achieve such a bound (see Section \ref{section:PCPvsBAsis} for a counterexample).
\smallskip

The rest of this paper is organized as follows:
\begin{itemize}
\item
Section \ref{section:FindingCurves} describes how we find a curve with trace $\pm t$ (Algorithm~1.1),\\ and how the primes $p_1,\ldots,p_n$ are selected (Algorithm~2.1).
\item
Section \ref{section:IsogenyVolcanoes} discusses isogeny volcanoes (Algorithms 1.2 and 1.3).
\item
Section \ref{section:PolycyclicPresentation} defines an optimal polycyclic presentation of $\clD$,\\ and gives an algorithm to compute one (Algorithm~2.2).
\item
Section \ref{section:CRT} addresses the CRT computations (Algorithms~2.3, 2.4, and 2.5).
\item
Section \ref{section:Complexity} contains a complexity analysis and proves Theorem~1.
\item
Section \ref{section:Performance} provides computational results.
\end{itemize}
Included in Section \ref{section:Performance} are timings obtained while constructing pairing-friendly curves of prime order over finite fields of cryptographic size.

\section{Finding an Elliptic Curve With a Given Number of Points}\label{section:FindingCurves}

Given a prime $p$ and a positive integer $t<2\sqrt{p}$, we seek an element of $\Ellt(\Fp)$, equivalently, an elliptic curve $E/\Fp$ with either $N_0=p+1-t$ or $N_1=p+1+t$ points.  This is essentially the problem considered in the introduction, but since we do not yet know $H_D$, we cannot apply the CM method.

Instead, we generate curves at random and test whether $\#E\in\{N_0,N_1\}$, where $\#E$ is the cardinality of the group $E(\Fp)$.  This test takes very little time, given the prime factorizations of $N_0$ and $N_1$, and does not require computing $\#E$.  However, in the absence of any optimizations we expect to test many curves: $2\sqrt{p}+O(1)$, on average, for fixed $p$ and varying $t$.  Factoring $N_0$ and $N_1$ is easy by comparison.

For the CRT-based algorithm in \cite{Belding:HilbertClassPolynomial}, searching for elements of $\Ellt(\Fp)$ dominates the computation.  In the example given there, this single step takes more than 50 times as long as the entire computation of $H_D$ using the floating-point method of~\cite{Enge:FloatingPoint}.  We address this problem here in detail, giving both asymptotic and constant factor improvements.  In aggregate, the improvements we suggest can reduce the time to find an element of $\Ellt(\Fp)$ by a factor of over 100; under the heuristic analysis of Section \ref{subsection:heuristics} this is no longer the asymptotically dominant step.

These improvements are enabled by a careful selection of primes $p\in\PD$, which is described in Section~\ref{subsection:PickPrimes}.  Contrary to what one might assume, the smallest primes in $\PD$ are not necessarily the best choices.  The expected time to find an element of $\Ellt(\Fp)$ can vary dramatically from one prime to the next, especially when one considers optimizations whose applicability may depend on $N_0$ and $N_1$.  In order to motivate our selection criteria, we first consider how we may narrow the search by our choice of~$p$, which determines $t=t(p)$ and therefore $N_0$ and $N_1$.

\subsection{The density of curves with trace $\boldsymbol{\pm t}$}\label{subsection:Density}
We may compute the density of $\Ellt(\Fp)$ as a subset of $\Fp$ via a formula of Deuring \cite{Deuring:CountingEC}.  For convenience we define
\begin{equation}\label{equation:density}
\rho(p,t) = \frac{H(4p-t^2)}{p} \approx \frac{\#\Ellt(\Fp)}{\#\Fp},
\end{equation}
where $H(4p-t^2)$ is the Hurwitz class number (as in \cite[Def.~5.3.6]{Cohen:CANT} or \cite[p.~319]{Cox:ComplexMultiplication}).  A more precise formula uses weighted cardinalities, but the difference is negligible, see
\cite[Thm.~14.18]{Cox:ComplexMultiplication} or \cite{Lenstra:ECM} for further details.

We expect to sample approximately $1/\rho(p,t)$ random curves over $\Fp$ in order to find one with trace $\pm t$.  When selecting primes $p\in\PD$, we may give preference to primes with larger $\rho$-values.  Doing so typically increase the average density by a factor of 3 or 4, compared to simply using the smallest primes in $\PD$.  It also makes $N_0$ and $N_1$ more likely to be divisible by small primes, which interacts favorably with the optimizations of the next section.

Using primes with large $\rho$-values improves the asymptotic results of Section~\ref{section:Complexity} by an $O(\log|D|)$ factor.  Effectively, we force the size of $\Ellt(\Fp)$ to increase with~$p$, even though the size of $\EllD(\Fp)$ is fixed at $h(D)$.  This process tends to favor primes in $\PD$ for which $v(p)$ has many small factors, something we must consider when enumerating $\EllD(\Fp)$ in Algorithm~1.3.

\subsection{Families with prescribed torsion}\label{subsection:PrescribedTorsion}
In addition to increasing the density of $\Ellt(\Fp)$ relative to $\Fp$, we can further accelerate our random search by sampling a subset of $\Fp$ in which $\Ellt(\Fp)$ has even greater density.  Specifically, we may restrict our search to a family of curves whose order is divisible by $m$, for some small $m$ dividing $N_0$ or $N_1$ (ideally both).  We have some control over $N_0$ and $N_1$ via our choice of $p\in\PD$, and in practice we find we can easily arrange for $N_0$ or $N_1$ to be divisible by a suitable $m$, discarding only a constant fraction of the primes in $\PD$ we might otherwise consider (making the primes we do use slightly larger).

To generate a curve whose order is divisible by $m$, we select a random point on $Y_1(m)/\Fp$ and construct the corresponding elliptic curve.
Here $Y_1(m)$ is the affine subcurve of the modular curve $X_1(m)$, which parametrizes elliptic curves with a point of order $m$.
We do this using plane models $F_m(r,s)=0$ that have been optimized for this purpose, see \cite{Sutherland:PrescribedTorsion}.
For $m$ in the set $\{2,3,4,5,6,7,8,9,10,12\}$, the curve $X_1(m)$ has genus 0, and we obtain Kubert's parametrizations \cite{Kubert:Torsion} of elliptic curves with a prescribed (cyclic) torsion subgroup over $\Q$.  Working in $\Fp$, we may use any $m$ not divisible $p$, although we typically use $m \le 40$, due to the cost of finding points on $F_m(r,s)=0$.

We augment this approach with additional torsion constraints that can be quickly computed.  For example, to generate a curve containing a point of order 132, it is much faster to generate several curves using $X_1(11)$ and apply tests for 3 and 4 torsion to each than it is to use $X_1(132)$.  A table of particularly effective combinations of torsion constraints, ranked by cost/benefit ratio, appears in Appendix~2.

The cost of finding points on $F_m(r,s)=0$ is negligible when $m$ is small, but grows with the genus (more precisely, the gonality) of $X_1(m)$, which is $O(m^2)$, by \cite[Thm.~1.1]{Jeon:ArithmeticModularCurves}.  For $m < 23$ the gonality is at most 4 (see Table~5 in \cite{Sutherland:PrescribedTorsion}), and points on $F_m(r,s)$ can be found quite quickly (especially when the genus is 0 or 1).

Provided that we select suitable primes from $\PD$, generating curves with prescribed torsion typically improves performance by a factor of 10 to 20.

\subsection{Selecting suitable primes}\label{subsection:PickPrimes}
We wish to select primes in $\mathcal{P}_D$ that maximize the benefit of the optimizations considered in Sections \ref{subsection:Density} and \ref{subsection:PrescribedTorsion}.
Our strategy is to enumerate a set of primes
\begin{equation}\label{equation:Sz}
S_z=\{p\in\mathcal{P}_D:1/\rho(p,t(p)) \le z\}
\end{equation}
that is larger than we need, and to then select a subset $S\subset S_z$ of the ``best" primes.  We require that $S$ be large enough to satisfy
$$
\sum_{p\in S}\lg p > b = \lg{B}+2,
$$
where $B$ is a bound on the coefficients of $H_D(X)$, obtained via Lemma~\ref{lemma:Bbound}, and ``$\lg$" denotes the binary logarithm.
We typically seek to make $S_z$ roughly 2 to 4 times the size of $S$, starting with a nominal value for $z$ and increasing it as required.

To enumerate $S_z$ we first note that if $4p=t^2-v^2D$ for some $p\in S_z$, then
$$\frac{1}{\rho(p,t)}=\frac{p}{H(4p-t^2)}=\frac{p}{H(-v^2D)}\le z.$$
Hence for a given $v$, we may bound the $p\in S_z$ with $v(p)=v$ by
\begin{equation}\label{equation:pbound}
p\le zH(-v^2D).
\end{equation}
To find such primes, we seek $t$ for which $p=(t^2-v^2D)/4$ is a prime satisfying~(\ref{equation:pbound}).  This is efficiently accomplished by \emph{sieving} the polynomial $t^2-v^2D$, see \cite[\S 3.2.6]{Crandall:PrimeNumbers}.
To bound $v=v(p)$ for $p\in S_z$, we note that $p > -v^2D/4$, hence
\begin{equation}\label{equation:vbound}
-v^2D < 4zH(-v^2D).
\end{equation}
For fixed $z$, this inequality will fail once $v$ becomes too large.  If we have
\begin{equation}\label{equation:v0bound}
\frac{v}{(\log\log(v+4))^2} \sge \frac{44zH(-D)}{-D},
\end{equation}
then (\ref{equation:vbound}) cannot hold, by Lemma~\ref{lemma:vbound} of Appendix~1.

\subsection*{Example}
Consider the construction of $S_z$ for $D=-108708$, for which we have $H(-D)=h(D)=100$.  We initially set $z$ to $-D/(2H(-D))\approx 543$.  For $v=1$ this yields the interval $[-v^2D/4,zH(-v^2D)]=[-D/4,-D/2]=[27177,54354]$, which we search for primes of the form $(t^2-D)/4$ by sieving $t^2-D$ with $t\le\sqrt{-2D}$, finding
17 such primes.  For $v=2$ we have $H(-v^2D)=300$ and search the interval $[-D,-3D/2]=[108708,163062]$ for primes of the form $(t^2-4D)/4$, finding 24 of them.  For $v=3$ we have $H(-v^2D)=400$ and the interval $[-9D/4,-2D]$ is empty.  The interval is also empty for $3<v<39$, and (\ref{equation:v0bound}) applies to all $v\ge 39$.

At this point $S_z$ is not sufficiently large, so we increase $z$, say by 50\%, obtaining $z \approx 814$.  This expands the intervals for $v=1,2$ and gives nonempty intervals for $v=3,4$, and we find an additional 74 primes.  Increasing $z$ twice more, we eventually reach $z\approx 1831$, at which point $S_z$ contains 598 primes with total size around 11911 bits.  This is more than twice $b=\lg{B}+2\approx 5943$, so we stop.  The largest prime in $S_z$ is $p=5121289$, with $v(p)=12$.
\smallskip

Once $S_z$ has been computed, we select $S\subset S_z$ by ranking the primes $p\in S_z$ according to their cost/benefit ratio.  The cost is the expected time to find a curve in $\Ellt(\Fp)$, taking into account the density $\rho(p,t)$ and the $m$-torsion constraints applicable to $N_0$ and $N_1$, and the benefit is $\lg p$, the number of bits in $p$.  Only a small set of torsion constraints are worth considering, and a table of these may be precomputed.  See Appendix~2 for further details.

The procedure for selecting primes is summarized below.  We assume that $h(D)$ has been obtained in the process of determining $B$ and $b=\lg{B}+2$, which allows $H(-D)$ and $\rho(p,t)$ to be easily computed (see (\ref{equation:vb1}) and (\ref{equation:vb2}) in Appendix~1).

\algstart{{\bf 2.1}}{Given $D$, $b$, and parameters $k>1$, $\delta>0$, select $S\subset\mathcal{P}_D$:}
\algitem
Let $z= -D/(2H(-D))$.
\algitem
Compute $S_z=\{p\in\mathcal{P}_D:1/\rho(p,t(p)) \le z\}$.
\algitem
If $\sum_{p\in S_z}\lg{p} \le kb$, then set $z\leftarrow (1+\delta)z$ and go to Step 2.
\algitem
Rank the primes in $S_z$ by increasing cost/benefit ratio as $p_1,\ldots,p_{n_z}$.
\algitem
Let $S=\{p_1,\ldots,p_n\}$, with $n\le n_z$ minimal subject to $\sum_{p\in S}\lg p > b$.
\algend
\vspace{-3pt}

In Step 3 we typically use $k=2$ or $k=4$ (a larger $k$ may find better primes), and $\delta=1/2$.  The complexity of Algorithm~2.1 is analyzed in Section \ref{section:Complexity}, where it is shown to run in expected time $O(|D|^{1/2+\epsilon})$, under the GRH (Lemma~\ref{lemma:sieve}).  This is negligible compared to the total complexity of $O(|D|^{1+\epsilon})$ and very fast in practice.  

In the $D=-108708$ example above, Algorithm~2.1 selects 313 primes in $S_z$, the largest of which is $p=4382713$, with $v=12$ and $t=1370$.  This largest prime is actually a rather good choice, due to the torsion constraints that may be applied to $N_0=p+1-t$, which is divisible by 3, 4, and~11.  We expect to test the orders of fewer than 40 curves for this prime, and on average need to test about 60 curves for each prime in $S$, fewer than 20,000 in all.

For comparison, the example in \cite[p.~294]{Belding:HilbertClassPolynomial} uses the least 324 primes in $\PD$,
the largest of which is only 956929, but nearly 500,000 curves are tested, over 1500 per prime.  The difference in running times is even greater, 0.2 seconds versus 18.5 seconds, due to optimizations in the testing algorithm of the next section.
\vspace{-3pt}

\subsection{Testing curves}\label{subsection:TestingCurves}

When $p$ is large, the vast majority of the random curves we generate will not have trace $\pm t$, even after applying the optimizations above.  To quickly filter a batch of, say, 50 or 100 curves, we pick a random point $P$ on each curve and simultaneously compute $(p+1)P$ and $tP$.
Here we apply standard multi-exponentiation techniques to scalar multiplication in $E(\Fp)$, using a precomputed NAF representation, see \cite[Ch.~9]{Cohen:HECHECC}.  We perform the group operations in parallel to minimize the cost of field inversions, using affine coordinates as in \cite[\S 4.1]{KedlayaSutherland:HyperellipticLSeries}.  We then test whether $(p+1)P=\pm tP$, as suggested in \cite{Belding:HilbertClassPolynomial}, and if this fails to hold we reject the curve, since its order cannot be $p+1\pm t$.

To each curve that passes this test, we apply the algorithm \textsc{TestCurveOrder}.
In the description below, $\mathcal{H}_p=[p+1-2\sqrt{p},p+1+2\sqrt{p}]$ denotes the Hasse interval, and the index $s\in\{0,1\}$ is used to alternate between $E$ and its quadratic twist $\tilde{E}$.

\algstart{\textsc{TestCurveOrder}}{Given an elliptic curve $E/\Fp$ and factored integers $N_0,N_1\in\mathcal{H}_p$ with $N_0< N_1$ and $N_0+N_1=2p+2${\rm :}}
\algitem
If $p\le 11$, return {\bf true} if $\#E\in\{N_0,N_1\}$ and {\bf false} otherwise.
\algitem
Set $E_0\leftarrow E$, $E_1\leftarrow \tilde{E}$, $m_0\leftarrow 1$, $m_1\leftarrow 1$, and $s\leftarrow 0$.
\algitem
Select a random point $P\in E_s(\Fp)$.
\algitem
Use \textsc{FastOrder} to compute the order $n_s$ of the point $Q=m_sP$, assuming $n_s$ divides $N_s/m_s$.
If this succeeds, set $m_s\leftarrow m_sn_s$ and proceed to Step 5.
If not, provided that $m_0|N_1$, $m_1|N_0$, and  $N_0 < N_1$, swap $N_0$ and $N_1$ and go to Step 3, but
otherwise return {\bf false}.
\algitem
Set $a_1\leftarrow 2p+2\bmod m_1$ and $\mathcal{N}\leftarrow \{m_0x:x\in\Z\}\cap\{m_1x+a_1:x\in\Z\}\cap\mathcal{H}_p$.\\
If $\mathcal{N}\subseteq\{N_0,N_1\}$ return {\bf true}, otherwise set $s\leftarrow 1-s$ and go to Step 3.
\algend

\textsc{TestCurveOrder} computes integers $m_s$ dividing $\#E_s$ by alternately computing the orders of random points on $E$ and $\tilde{E}$. If an order computation fails (this happens when $n_s\nmid N_s/m_s$), it rules out $N_s$ as a possibility for $\#E$.  If both $N_0$ and $N_1$ are eliminated, the algorithm returns {\bf false}.  Otherwise a divisor $n_s$ of $N_s$ is obtained and the algorithm continues until it narrows the possibilities for $\#E$ to a nonempty subset of $\{N_0,N_1\}$ (it need not determine which).  
The set $\mathcal{N}$ computed in Step 5 must contain $\#E$, since $m_0$ divides $\#E$ and $m_1$ divides $\#\tilde{E}$ (the latter implies $\#E\equiv 2p+2\bmod m_1$, since $\#E+\#\tilde{E}=2p+2$).  The complexity of the algorithm (and a proof that it terminates) is given by Lemma~\ref{lemma:TestCurve} of Section~\ref{section:Complexity}.
\smallskip

A simple implementation of \textsc{FastOrder} appears below, based on a recursive algorithm to compute the order of a generic group element due to Celler and Leedham-Green \cite{Celler:GLOrder}.
By convention, generic groups are written multiplicatively, and we do so here, although we apply \textsc{FastOrder} to the additive groups $E(\Fp)$ and $\tilde{E}(\Fp)$. The function $\omega(N)$ counts the distinct prime factors of $N$.

\algstart{\textsc{FastOrder}}{Given an element $\alpha$ of a generic group $G$ and a factored integer $N$, compute the  function $\mathcal{A}(\alpha,N)$, defined to be the factored integer $M=|\alpha|$ when $M$ divides $N$, and $0$ otherwise.}
\algitem
If $N$ is a prime power $p^n$, compute $\alpha^{p^i}$ for increasing $i$ until the identity is reached (in which case return $p^i$), or $i=n$ (in which case return 0).
\algitem
Let $N=N_1N_2$ with $N_1$ and $N_2$ coprime and $|\omega(N_1)-\omega(N_2)|\le 1$.\\
Recursively compute $M=\mathcal{A}(\alpha^{N_2},N_1)\cdot\mathcal{A}(\alpha^{N_1},N_2)$ and return $M$.
\algend
\noindent
This algorithm uses $O(\log{N}\log\log{N})$ multiplications (and identity tests) in $G$. A slightly faster algorithm  \cite[Alg.~7.4]{Sutherland:Thesis} is used in the proof of Theorem~1.   
In practice, the implementation of \textsc{TestCurveOrder} and \textsc{FastOrder} is not critical, since most of the time is actually spent performing the scalar multiplications discussed above (these occur in Step 3 of Algorithm 1.1 below).
\smallskip

We now give the complete algorithm to find an element of $\Ellt(\Fp)$.
For reasons discussed in the next section, we exclude the $j$-invariants 0 and 1728.

\algstart{{\bf 1.1}}{Given $p \in \PD$, find $j\in\Ellt(\Fp)-\{0,1728\}$.}
\algitem
Factor $N_0=p+1-t$ and $N_1=p+1+t$, and choose torsion constraints.
\algitem
Generate a batch of random elliptic curves $E_i/\Fp$ with $j(E_i)\notin\{0,1728\}$ that satisfy these constraints and pick a random point $P_i$ on each curve.
\algitem
For each $i$ with $(p+1)P_i=\pm tP_i$, test whether $\#E_i\in \{N_0,N_1\}$ by calling \textsc{TestCurveOrder}, using the factorizations of $N_0$ and $N_1$.
\algitem
If $\#E_i\in\{N_0,N_1\}$ for some $i$, output $j(E_i)$, otherwise return to Step 2.
\algend
\noindent
The torsion constraints chosen in Step 1 may be precomputed by Algorithm~2.1 in the process of selecting $S\subset\PD$.  In Step 2 we may generate $E_i$ with $m$-torsion as described in Section \ref{subsection:PrescribedTorsion}; as a practical optimization, if $X_1(m)$ has genus 0 we generate both $E_i$ and $P_i$ using the parametrizations in \cite{Atkin:ECMCurves}.  In Step 3 the point~$P_i$ can also be used as the first random point chosen in \textsc{TestCurveOrder}.  The condition $(p+1)P_i=\pm tP_i$ is tested by performing scalar multiplications in parallel, as described above;
when torsion constraints determine the sign of $t$, we instead test whether $(p+1-t)P_i=0$ or $(p+1+t)P_i=0$, as appropriate.

\section{Isogeny Volcanoes}\label{section:IsogenyVolcanoes}
The previous section addressed the first step in computing $H_D\bmod p$: finding an element of $\Ellt(\Fp)$.
In this section we address the next two steps: finding an element of $\EllD(\Fp)$ and enumerating $\EllD(\Fp)$.  This yields the roots of $H_D\bmod p$.  We utilize the graph of $\ell$-isogenies defined on $\Ellt(\Fp)$.  We regard this as an undirected graph, noting that the dual isogeny \cite[\S III.6]{Silverman:EllipticCurves1} lets us traverse edges in either direction.  We permit self-loops in our graphs but not multiple edges.

\begin{definition}
Let $\ell$ be prime.  An $\ell$-\emph{volcano} is an undirected graph with vertices partitioned into levels $V_0,\ldots,V_d$, in which the subgraph on $V_0$ (the \emph{surface}) is a regular connected graph of degree at most 2, and also:
\begin{enumerate}
\item
For $i>0$, each vertex in $V_i$ has exactly one edge leading to a vertex in $V_{i-1}$,\\
and every edge not on the surface is of this form.
\item
For $i<d$, each vertex in $V_i$ has degree $\ell+1$.
\end{enumerate}
\end{definition}

The surface $V_0$ of an $\ell$-volcano is either a single vertex (possibly with a self-loop), two vertices connected by an edge, or a (simple) cycle on more than two vertices, which is the typical case.  We call $V_d$ the \emph{floor} of the volcano, which coincides with the surface when $d=0$.  For $d>0$ the vertices on the floor have degree 1, and in every case their degree is at most 2; all other vertices have degree $\ell+1 > 2$.

We refer to $d$ as the \emph{depth} of the $\ell$-volcano.  The term ``height" is also used \cite{Miret:VolcanoHeight}, but ``depth" better suits our indexing of the levels $V_i$ and is consistent with \cite{Kohel:thesis}.
\medskip

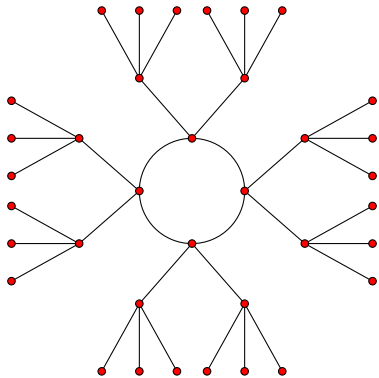
\begin{figure}[htp]
\centering
\begin{tikzpicture}
\draw (0,0) circle (0.7);
\draw[black] (0,-0.7) -- (-0.7,-1.5);
\draw[black] (0,-0.7) -- (0.7,-1.5);
\draw[black] (0,0.7) -- (-0.7,1.5);
\draw[black] (0,0.7) -- (0.7,1.5);
\draw[black] (-0.7,0) -- (-1.5,-0.7);
\draw[black] (-0.7,0) -- (-1.5,0.7);
\draw[black] (0.7,0) -- (1.5,-0.7);
\draw[black] (0.7,0) -- (1.5,0.7);
\draw[black] (-0.7,-1.5) -- (-0.2,-2.4);
\draw[black] (-0.7,-1.5) -- (-0.7,-2.4);
\draw[black] (-0.7,-1.5) -- (-1.2,-2.4);
\draw[black] (-0.7,1.5) -- (-0.2,2.4);
\draw[black] (-0.7,1.5) -- (-0.7,2.4);
\draw[black] (-0.7,1.5) -- (-1.2,2.4);
\draw[black] (0.7,-1.5) -- (0.2,-2.4);
\draw[black] (0.7,-1.5) -- (0.7,-2.4);
\draw[black] (0.7,-1.5) -- (1.2,-2.4);
\draw[black] (0.7,1.5) -- (0.2,2.4);
\draw[black] (0.7,1.5) -- (0.7,2.4);
\draw[black] (0.7,1.5) -- (1.2,2.4);
\draw[black] (-1.5,-0.7) -- (-2.4,-0.2);
\draw[black] (-1.5,-0.7) -- (-2.4,-0.7);
\draw[black] (-1.5,-0.7) -- (-2.4,-1.2);
\draw[black] (-1.5,0.7) -- (-2.4,0.2);
\draw[black] (-1.5,0.7) -- (-2.4,0.7);
\draw[black] (-1.5,0.7) -- (-2.4,1.2);
\draw[black] (1.5,-0.7) -- (2.4,-0.2);
\draw[black] (1.5,-0.7) -- (2.4,-0.7);
\draw[black] (1.5,-0.7) -- (2.4,-1.2);
\draw[black] (1.5,0.7) -- (2.4,0.2);
\draw[black] (1.5,0.7) -- (2.4,0.7);
\draw[black] (1.5,0.7) -- (2.4,1.2);
\draw[fill=red] (0,0.7) circle (0.05);
\draw[fill=red] (0,-0.7) circle (0.05);
\draw[fill=red] (0.7,0) circle (0.05);
\draw[fill=red] (-0.7,0) circle (0.05);
\draw[fill=red] (-0.7,-1.5) circle (0.05);
\draw[fill=red] (-0.7,1.5) circle (0.05);
\draw[fill=red] (0.7,-1.5) circle (0.05);
\draw[fill=red] (0.7,1.5) circle (0.05);
\draw[fill=red] (-1.5,-0.7) circle (0.05);
\draw[fill=red] (-1.5,0.7) circle (0.05);
\draw[fill=red] (1.5,-0.7) circle (0.05);
\draw[fill=red] (1.5,0.7) circle (0.05);
\draw[fill=red] (-2.4,-0.2) circle (0.05);
\draw[fill=red] (-2.4,-0.7) circle (0.05);
\draw[fill=red] (-2.4,-1.2) circle (0.05);
\draw[fill=red] (-2.4,0.2) circle (0.05);
\draw[fill=red] (-2.4,0.7) circle (0.05);
\draw[fill=red] (-2.4,1.2) circle (0.05);
\draw[fill=red] (2.4,-0.2) circle (0.05);
\draw[fill=red] (2.4,-0.7) circle (0.05);
\draw[fill=red] (2.4,-1.2) circle (0.05);
\draw[fill=red] (2.4,0.2) circle (0.05);
\draw[fill=red] (2.4,0.7) circle (0.05);
\draw[fill=red] (2.4,1.2) circle (0.05);
\draw[fill=red] (-0.2,-2.4) circle (0.05);
\draw[fill=red] (-0.7,-2.4) circle (0.05);
\draw[fill=red] (-1.2,-2.4) circle (0.05);
\draw[fill=red] (0.2,-2.4) circle (0.05);
\draw[fill=red] (0.7,-2.4) circle (0.05);
\draw[fill=red] (1.2,-2.4) circle (0.05);
\draw[fill=red] (-0.2,2.4) circle (0.05);
\draw[fill=red] (-0.7,2.4) circle (0.05);
\draw[fill=red] (-1.2,2.4) circle (0.05);
\draw[fill=red] (0.2,2.4) circle (0.05);
\draw[fill=red] (0.7,2.4) circle (0.05);
\draw[fill=red] (1.2,2.4) circle (0.05);
\end{tikzpicture}
\caption{A 3-volcano of depth 2, with a 4-cycle on the surface.}
\end{figure}
\vspace{-6pt}

\begin{definition}
For a prime $\ell\ne p$, let $\Glt(\Fp)$ be the undirected graph with vertex set $\Ellt(\Fp)$ that contains the edge $(j_1,j_2)$ if and only if $\Phi_\ell(j_1,j_2)=0$.
\end{definition}

Here $\Phi_\ell$ denotes the classical modular polynomial.  With at most two exceptions, the components of $\Glt(\Fp)$ are $\ell$-volcanoes.
The level at which $j(E)\in\Ellt(\Fp)$ resides in its $\ell$-volcano is determined by the power of $\ell$ dividing the conductor of $\End(E)$.

The discriminant $D$ may be written as $D=u^2D_K$, where $D_K$ is the discriminant of the maximal order $\O_K$ containing $\O$, and $u=[\O_K:\O]$ is the conductor of $\O$.  We also have the discriminant
\begin{equation}
D_\pi=t^2-4p=v^2D=w^2D_K
\end{equation}
of the order $\Z[\pi]\subseteq\O_K$ with conductor $w=uv$, generated by the Frobenius
endomorphism $\pi$ with trace $t$ (note $\pi=\pi_E$ for all $j(E)\in \Ellt(\Fp)$).
The order $\O$ contains $\Z[\pi]$, and for any $j(E)\in\Ellt(\Fp)$ we have $\Z[\pi]\subseteq\End(E)\subseteq\O_K$.
Curves with $\End(E)\cong\Z[\pi]$ lie on the floor of their $\ell$-volcano, while those with $\End(E)\cong\O_K$ lie on the surface.  More generally, the following proposition holds.

\begin{proposition}\label{prop:volcano}
Let $p\in\mathcal{P}_D$ and let $\ell\ne p$ be a prime.  The components of $\Glt(\Fp)$ that do not contain $j=0,1728$ are $\ell$-volcanoes of depth $d=\nu_\ell(w)$.  Each has an associated order~$\O_0$, with $\Z[\pi]\subseteq\O_0\subseteq\O_K$ and $\ell\nmid[\O_K:\O_0]$, and we have
$$j(E)\in V_i \quad\Longleftrightarrow\quad\End(E)\cong\O_i,$$
where $\O_i$ is the order of index $\ell^i$ in $\O_0$.
\end{proposition}
Here $\nu_\ell$ denotes the $\ell$-adic valuation (so $\ell^d|w$ but $\ell^{d+1}\nmid w$).
The proposition follows essentially from \cite[Prop.~23]{Kohel:thesis}.  See \cite[Lemmas 2.1-6]{Fouquet:IsogenyVolcanoes} for additional details and \cite[Thm.~1.19, Prop.~12.20]{Washington:EllipticCurves} for properties of $\Phi_\ell$.

We have excluded $j=0, 1728$ (which can arise only when $D_K=-3,-4$) for technical reasons, see \cite[Rem.~12.21]{Washington:EllipticCurves}.  However a nearly equivalent statement holds; only the degrees of the vertices 0 and 1728 are affected.

\subsection{Obtaining an element of $\boldsymbol{\EllD(\Fp)}$}\label{subsection:FindEllD}

Given $j(E)\in\Ellt(\Fp)-\{0,1728\}$, we may apply Proposition~\ref{prop:volcano} to obtain an element of $\EllD(\Fp)$.
Let $u$ and $u_E$ be the conductors of $\O$ and $\End(E)$ respectively; both $u$ and $u_E$ divide $w$, the conductor of $D_\pi=t^2-4p$.
Suppose $\nu_\ell(u_E)\not=\nu_\ell(u)$ for some prime $\ell$.
If we replace $j=j(E)$ by a vertex at level $\nu_\ell(u)$ in $j$'s $\ell$-volcano, we then have $\nu_\ell(u_E)=\nu_\ell(u)$.
Proposition~\ref{prop:volcano} assures us that this ``adjustment" only affects the power of $\ell$ dividing $u_E$.  Repeating this for each prime $\ell|w$, we eventually have $u_E=u$ and $j(E)\in\EllD(\Fp)$.

To change location in an $\ell$-volcano we walk a \emph{path}, which we define to be a sequence of vertices $j_0,\ldots,j_n$ connected by edges $(j_k,j_{k+1})$, such that $j_{k-1}\ne j_{k+1}$ for all $0<k<n$ (this condition is enforced by never taking a backward step).

Paths in $\Glt(\Fp)$ are computed by choosing an initial edge $(j_0,j_1)$, and for $k>0$ extending the path $j_0,\ldots,j_k$ by picking a root $j_{k+1}$ of the polynomial
\begin{equation*}
f(X) = \Phi_\ell(X,j_k)/(X-j_{k-1})^e\in\Fp[x].
\end{equation*}
Here $e$ is the multiplicity of the root $j_{k-1}$ in $\Phi_\ell(X,j_k)$, equal to one in all but a few special cases (see \cite[Lemma~2.6 and Thm.~2.2]{Fouquet:IsogenyVolcanoes}).  If $f(X)$ has no roots in $\Fp$, then  $j_k$ has no neighbors other than $j_{k-1}$ and the path must end at $j_k$.

When a path has $j_k\in V_i$ and $j_{k+1}\in V_{i+1}$, we say the path \emph{descends} at $k$.  Once a path starts descending, it must continue to do so.  If a path descends at every step and terminates at the floor, we call it a \emph{descending path}, as in \cite[Def.~4.1]{Fouquet:IsogenyVolcanoes}.
\smallskip

We now present an algorithm to determine the level of a vertex $j$ in an $\ell$-volcano, following Kohel \cite[p.~46]{Kohel:thesis}.  
When walking a path, we suppose neighbors are picked uniformly at random whenever there is a choice to be made.

\algstart{\textsc{FindLevel}}
{Compute the level of $j$ in an $\ell$-volcano of depth $d$$:$}
\algitem
If $\deg(j)\ne \ell+1$ then return $d$, otherwise let $j_1\ne j_2$ be neighbors of $j$.
\algitem
Walk a path of length $k_1\le d$ extending $(j,j_1)$.
\algitem
Walk a path of length $k_2\le k_1$ extending $(j,j_2)$.
\algitem
Return $d-k_2$.
\algend
\noindent
If \textsc{FindLevel} terminates in Step 1, then $j$ is on the floor at level $d$.
The paths walked in Steps 2 and 3 are extended as far as possible, up to the specified bound. If $j$ is on the surface, then these paths both have length $d$, and otherwise at least one of them is a descending path of length $k_2$.  In both cases, $j$ is on level $d-k_2$.
\smallskip

We use the algorithms below to change levels in an $\ell$-volcano of depth $d > 0$.

\algstart{\textsc{Descend}}
{Given $j\in V_{k}\ne V_d$, return $j'\in V_{k+1}$$:$}
\algitem
If $k=0$, walk a path $(j=j_0,\ldots,j_n)$ to the floor and return $j'=j_{n-d+1}$.
\algitem
Otherwise, let $j_1$ and $j_2$ be distinct neighbors of $j$.
\algitem
Walk a path of length $d-k$ extending $(j,j_1)$ and ending in $j^*$.
\algitem
If $\deg(j^*)=1$ then return $j'=j_1$, otherwise return $j'=j_2$.
\algend

\algstart{\textsc{Ascend}}
{Given $j\in V_{k}\ne V_0$, return $j'\in V_{k-1}$$:$}
\algitem
If $\deg(j)=1$ then let $j'$ be the neighbor of $j$ and return $j'$,\\
otherwise let $j_1,\ldots,j_{\ell+1}$ be the neighbors of $j$.
\algitem
For each $i$ from 1 to $\ell$:
\begin{enumerate}
\item
Walk a path of length $d-k$ extending $(j,j_i)$ and ending in $j^*$.
\item
If $\deg(j^*)>1$ then return $j'=j_i$.
\end{enumerate}
\algitem
Return $j'=j_{\ell+1}$.
\algend
\noindent
The correctness of \textsc{Descend} and \textsc{Ascend} is easily verified.  We note that if $k=0$ in \textsc{Descend}, then the expected value of $n$ is at most $d+2$ (for any $\ell$).
\medskip

We now give the algorithm to find an element $j'\in\EllD(\Fp)$, given $j\in\Ellt(\Fp)$.  We use a bound $L$ on the primes $\ell|w$, reverting to a computation of the endomorphism ring to address $\ell>L$, as discussed below.  This is never necessary when $D$ is fundamental, but may arise when the conductor of $D$ has a large prime factor.

\algstart{{\bf 1.2}}{Let $p\in\mathcal{P}_D$, let $u$ be the conductor of $D$, and let $w=uv$, where $v=v(p)$.
Given $j\in \Ellt(\Fp)-\{0,1728\}$, find $j'\in \EllD(\Fp)$$:$}
\algitem
For each prime $\ell|w$ with $\ell\le L=\max(\log|D|,v)$:
\begin{enumerate}
\item
Use \textsc{FindLevel} to determine the level of $j$ in its $\ell$-volcano.
\item
Use \textsc{Descend} and \textsc{Ascend} to obtain $j'$ at level $\nu_\ell(u)$ and set $j\leftarrow j'$.
\end{enumerate}
\algitem
If $u$ is not $L$-smooth, verify that $j\in \EllD(\Fp)$ and \textbf{\tt abort} if not.
\algitem
Return $j'=j$.
\algend
\noindent
The verification in Step 2 involves computing $\End(E)$ for an elliptic curve $E/\Fp$ with $j(E)=j$.  Here we may use the algorithm in \cite{BissonSutherland:Endomorphism}, or Kohel's algorithm \cite{Kohel:thesis}.  The former is faster in practice (with a heuristically subexponential running time) but for the proof of Theorem~1 we use the $O(p^{1/3})$ complexity bound of Kohel's algorithm, which depends only on the GRH.

For $p\in S$, we expect $v$ to be small, $O(\log^{3+\epsilon}|D|)$ under the GRH, and heuristically $O(\log^{1/2}|D|)$.  Provided $u$ does not contain a prime larger than $L$, the running time of Algorithm~1.2 is polynomial in $\log|D|$, under the GRH.

However, if $u$ is divisible by a prime $\ell>L$, we want to avoid the cost of computing $\ell$-isogenies.  Such an $\ell$ cannot divide $v$ (since $L\ge v$), so our desired $j'$ must lie on the floor of its $\ell$-volcano.  When $\ell$ is large, it is highly probable that our initial $j$ is already on the floor (this is where most of the vertices in an $\ell$-volcano lie), and this will still hold in Step 2.  Since $L\ge\log|D|$ is asymptotically larger than the number of prime factors of $u$, the probability of a failure in Step 2 is $o(1)$.  If Algorithm~1.2 aborts, we call Algorithm~1.1 again and retry.

If $D_K$ is $-3$ or $-4$, then $j$ may lie in a component of $\Glt(\Fp)$ containing $0$ or 1728.  However, provided we never pick $0$ or $1728$ when choosing a neighbor, \textsc{FindLevel}, \textsc{Descend}, and \textsc{Ascend} will correctly handle this case.

\subsection{Enumerating $\boldsymbol{\EllD(\Fp)}$}\label{subsection:EnumEllD}

Having obtained $j_0\in\EllD(\Fp)$, we now wish to enumerate the rest of $\EllD(\Fp)$.  We assume $h(D) > 1$ and apply the group action of $\clD$ to the set $\EllD(\Fp)$.  Let $\ell$ be a prime not dividing the conductor $u$ of $D$ with $\inkron{D}{\ell}\ne -1$.  Then $\ell$ can be uniquely factored in $\O$ into conjugate prime ideals as $(\ell)=\mathfrak{a}\bar{\mathfrak{a}}$, where $\mathfrak{a}$ and $\bar{\mathfrak{a}}$ both have norm $\ell$.  The ideals $\mathfrak{a}$ and $\bar{\mathfrak{a}}$ are distinct when $\inkron{D}{\ell}=1$, and in any case the ideal classes $[\mathfrak{a}]$ and $[\bar{\mathfrak{a}}]$ are inverses.  
The orders of $[\mathfrak{a}]$ and $[\bar{\mathfrak{a}}]$ in $\clD$ are equal, and we denote their common value by $\ord_D(\ell)$.  The following
proposition follows immediately from Propositions \ref{prop:CM} and \ref{prop:volcano}.

\begin{proposition}
Let $\ell\ne p$ be a prime such that $\ell\nmid u$ and $\inkron{D}{\ell}\ne -1$.  Then every element of $\EllD(\Fp)$ lies on the surface $V_0$ of its $\ell$-volcano, and $\#V_0=\ord_D(\ell)$.
\end{proposition}

If $\ord_D(\ell)=h(D)$, then $\EllD(\Fp)$ is equal to the surface of the $\ell$-volcano containing $j_0$, but in general we must traverse several volcanoes to enumerate $\EllD(\Fp)$.  We first describe how to walk a path along the surface of a single $\ell$-volcano.

When $\ell$ does not divide $v$, every $\ell$-volcano in $\Glt(\Fp)$ has depth zero.  In this case walking a path on the surface is trivial: for $\#V_0>2$ we choose one of the two roots of $\Phi_\ell(X,j_0)$, and every subsequent step is determined by the single root of the polynomial $f(X)=\Phi_\ell(X,j_i)/(X-j_{i-1})$.  The cost of each step is then
\begin{equation}\label{equation:isogenycost}
O(\ell^2+\M(\ell)\log{p})
\end{equation}
operations in $\Fp$, where $\M(n)$ is the complexity of multiplication (the first term is the time to evaluate $\Phi_\ell(X,j_i)$, the second term is the time to compute $X^p\bmod f$).

While it is simpler to restrict ourselves to primes $\ell\nmid v$ (there are infinitely many $\ell$ we might use), as a practical matter, the time spent enumerating $\EllD(\Fp)$ depends critically on $\ell$.  Consider $\ell=2$ versus $\ell=7$.  The cost of finding a root of $f(X)$ when $f$ has degree 7 may be 10 or 20 times the cost when $f$ has degree 2.
We much prefer $\ell=2$, even when the 2-volcano has depth $d>0$ (necessarily the case when $\inkron{D}{2}=1$).  The following algorithm allows us to handle $\ell$-volcanoes of any depth.

\algstart{\textsc{WalkSurfacePath}}{Given $j_0\in V_0$ in an $\ell$-volcano of depth $d$ and a positive integer $n< \#V_0$, return a path $j_0,j_1\ldots,j_n$ contained in $V_0$:}
\algitem
If $\deg(j_0)=1$ then return the path $j_0,j_1$, where $j_1$ is the neighbor of $j_0$.\\
Otherwise, walk a path $j_0,\ldots,j_d$ and set $i\leftarrow 0$.
\algitem
While $\deg(j_{i+d})=1$, replace $j_{i+1},\ldots,j_{i+d}$ by extending the path $j_0,\ldots,j_i$ by
$d$ steps, starting from a random unvisited neighbor $j_{i+1}'$ of $j_i$.
\algitem
Extend the path $j_0,\ldots,j_{i+d}$ to $j_0,\ldots,j_{i+d+1}$, then set $i\leftarrow i+1$.
\algitem
If $i=n$ then return $j_0,\ldots, j_n$, otherwise go to Step 2.
\algend
\noindent
When $d=0$ the algorithm necessarily returns a path that is contained in $V_0$.  Otherwise, the path extending $d+1$ steps beyond $j_i\in V_0$ in Step~3 guarantees that $j_{i+1}\in V_0$.  The algorithm maintains (for the current value of $i$) a list of visited neighbors of $j_i$ to facilitate the choice of an unvisited neighbor in Step~2.

To bound the expected running time, we count the vertices examined during its execution, that is, the number of vertices whose neighbors are computed.

\begin{proposition}\label{prop:WalkSurface}
Let the random variable $X$ be the number of vertices examined by \textsc{WalkSurfacePath}.
If $\#V_0=2$ then ${\rm\bf E}[X]= d+1+ld/2$, and otherwise
$${\rm\bf E}[X] \le d+(1+(\ell-1)d/2)n.$$
\end{proposition}
\begin{proof}
If $d=0$ then \textsc{WalkSurfacePath} examines exactly $n$ vertices and the proposition holds, so we assume $d > 0$ and note that $\deg(j_0)>1$ in this case.
We partition the execution of the algorithm into phases, with phase -1 consisting of Step 1, and the remaining phases corresponding to the value of $i$.  At the start of phase $i\ge 0$ we have $j_i\in V_0$ and the path $j_0,\ldots, j_{i+d}$.  Let the random variable $X_i$ be the number of vertices examined in phase $i$, so that $X=X_{-1}+X_0+\cdots+X_n$.  We have $X_{-1}=d$ and $X_n=0$.
For $0\le i < n$ we have $X_i=1+md$, where $m$ counts the number of incorrect choices of $j_{i+1}$ (those not in $V_0$).

We first suppose $\#V_0=2$.  In this case exactly one of the $\ell+1$ neighbors of $j_0$ lies in $V_0$.  Conditioning on $m$ we obtain
$${\rm\bf E}[X_0] = \sum_{m=0}^\ell\Bigl(1+md\Bigr)\frac{1}{\ell+1-m}\prod_{k=0}^{m-1}\left(\frac{\ell-k}{\ell+1-k}\right)
=\sum_{m=0}^\ell\frac{1+md}{\ell+1}=1+ld/2.
$$
This yields
$${\rm\bf E}[X]={\rm\bf E}[X_{-1}]+{\rm\bf E}[X_0]+{\rm\bf E}[X_1] = d + 1 + ld/2,$$
as desired.  We now assume $\#V_0>2$.  Then two of $j_0$'s neighbors lie in $V_0$ and we find that ${\rm\bf E}[X_0] = 1+(\ell-1)d/3$.
For $i>1$ we exclude the neighbor $j_{i-1}$ of $j_i$ and obtain ${\rm\bf E}[X_i] = 1+(\ell-1)d/2$.
Summing expectations completes the proof.
\end{proof}

Using an estimate of the time to find the roots of a polynomial of degree $\ell$ in $\Fp[X]$, we may apply Proposition~\ref{prop:WalkSurface} to optimize the choice of the primes $\ell$ that we use when enumerating $\EllD(\Fp)$, as discussed in the next section.  As an example, if $\inkron{D}{2}=1$ and $\nu_2(v)=2$, then we need to solve an average of roughly 2 quadratic equations for each vertex when we walk a path along the surface of a 2-volcano in $\Glt(\Fp)$.  This is preferable to using any $\ell>2$, even when $\ell\nmid v$.  On the other hand, if $\inkron{D}{5}=\inkron{D}{7}=1$ and $5|v$ but $7\nmid v$, we likely prefer $\ell=7$ to $\ell=5$.

\smallskip
We now present Algorithm~1.3, which, given $j_0\in\EllD(\Fp)$ and suitable lists of primes $\ell_i$ and integers $r_i$, outputs the elements of $\EllD(\Fp)-\{j_0\}$.  It may be viewed as a generalization of \textsc{WalkSurfacePath} to $k$ dimensions.

\algstart{1.3}{Given $j_0\in\EllD(\Fp)$, primes $\ell_1,\ldots,\ell_k$ with $\ell_i\nmid u$ and $\inkron{D}{\ell_i}\ne -1$, and integers $r_1,\ldots,r_k$, with $1< r_i\le \ord_D(\ell_i)$:}
\algitem
Use \textsc{WalkSurfacePath} to compute a path $j_0,j_1,\ldots,j_{r_k-1}$ of length $r_k-1$ on the surface of the $\ell_k$-volcano containing $j_0$, and output $j_1,\ldots,j_{r_k-1}$.
\algitem
If $k>1$ then for $i$ from 0 to $r_k-1$ recursively call Algorithm~1.3 using $j_i$, the primes $\ell_1,\ldots,\ell_{k-1}$, and the integers $r_1,\ldots,r_{k-1}$.
\algend

Proposition~\ref{prop:volcano} implies that Algorithm~1.3 outputs a subset of $\EllD(\Fp)$, since $j_0,j_1,\ldots,j_{r_k-1}$ all lie on the surface of the same $\ell_k$-volcano (and this applies recursively).  To ensure that Algorithm~1.3 outputs all the elements of $\EllD(\Fp)-\{j_0\}$, we use a polycyclic presentation for $\clD$, as defined in the next section.

\pagebreak
\section{Polycyclic Presentations of Finite Abelian Groups}\label{section:PolycyclicPresentation}

To obtain suitable sequences $\ell_1,\ldots,\ell_k$ and $r_1\ldots,r_k$ for use with Algorithm~1.3, we apply the theory of polycyclic presentations \cite[Ch.~8]{Holt:CGTHandbook}.  Of course $\clD$ is a finite abelian group, but the concepts we need have been fully developed in the setting of polycyclic groups, and conveniently specialize to the finite abelian case.

Let $\vec{\alpha}=(\alpha_1,\ldots,\alpha_k)$ be a sequence of generators for a finite abelian group $G$, and let
$G_i=\langle\alpha_1,\ldots,\alpha_i\rangle$ be the subgroup generated by $\alpha_1,\ldots,\alpha_i$.  The series
\begin{equation*}
1=G_0\le G_1 \le \cdots\le G_{k-1} \le G_{k} = G,
\end{equation*}
is necessarily a \emph{polycyclic series}, that is, a subnormal series in which each quotient $G_i/G_{i-1}$ is a cyclic group.
Indeed, $G_i/G_{i-1}=\langle\alpha_iG_{i-1}\rangle$, and $\vec{\alpha}$ is a \emph{polycyclic sequence} for $G$.  We say that $\vec{\alpha}$ is \emph{minimal} if none of the quotients are trivial.

When $G=\prod\langle\alpha_i\rangle$, we have $G_i/G_{i-1}\cong\langle\alpha_i\rangle$ and call $\vec{\alpha}$ a \emph{basis} for $G$, but this is a special case.  For abelian groups, $G_i/G_{i-1}$ is isomorphic to a subgroup of $\langle\alpha_i\rangle$, but it may be a proper subgroup, even when $\vec{\alpha}$ is minimal.

The sequence $r(\vec{\alpha})=(r_1,\ldots,r_k)$ of \emph{relative orders} for $\vec{\alpha}$ is defined by
\begin{equation*}
r_i=|G_i:G_{i-1}|.
\end{equation*}
We necessarily have $\prod r_i=|G|$, and if $\vec{\alpha}$ is minimal then $r_i>1$ for all $i$.  The sequences $\vec{\alpha}$ and $r(\vec{\alpha})$ allow us to uniquely represent every $\beta\in G$ in the form $$\beta=\vec{\alpha}^\vec{x}=\alpha_1^{x_1}\cdots\alpha_k^{x_k}.$$
\begin{lemma}\label{lemma:urep}
Let $\vec{\alpha}=(\alpha_1,\ldots,\alpha_k)$ be a sequence of generators for a finite abelian group $G$, let $r(\vec{\alpha})=(r_1,\ldots,r_k)$, and let $X(\vec{\alpha})=\{\vec{x}\in\Z^k:0\le x_i < r_i\}$.
\begin{enumerate}
\item
For each $\beta\in G$ there is a unique $\vec{x}\in X(\vec{\alpha})$ such that $\beta=\vec{\alpha}^\vec{x}$.
\item
The vector $\vec{x}$ such that $\alpha_i^{r_i}=\vec{\alpha}^\vec{x}$ has $x_j=0$ for $j\ge i$.
\end{enumerate}
\end{lemma}
\begin{proof}
See Lemmas 8.3 and 8.6 in \cite{Holt:CGTHandbook}.
\end{proof}
\noindent
The vector $\vec{x}$ is the \emph{discrete logarithm} (exponent vector) of $\beta$ with respect to~$\vec{\alpha}$.
The relations $\alpha_i^{r_i}=\vec{\alpha}^\vec{x}$ are called \emph{power relations}, and may be used to define a (consistent) polycyclic presentation for an abelian group $G$, as in \cite[Def.~8.7]{Holt:CGTHandbook}.

We now show that a minimal polycyclic sequence for $\clD$ provides suitable inputs for Algorithm~1.3.

\begin{proposition}\label{prop:enum}
Let $\vec{\alpha}=(\alpha_1,\ldots,\alpha_k)$ be a minimal polycyclic sequence for $\clD$ with relative orders $r(\vec{\alpha})=(r_1,\ldots,r_k)$, and let $\ell_1,\ldots,\ell_k$ be primes for which $\alpha_i$ contains an invertible ideal of norm $\ell_i$.  Given $j_0\in\EllD(\Fp)$, the primes $\ell_i$, and the integers $r_i$, {\rm Algorithm} $1.3$ outputs each element of $\EllD(\Fp)-\{j_0\}$ exactly once.
\end{proposition}
\begin{proof}
As previously noted, Proposition~\ref{prop:volcano} implies that the outputs of Algorithm~1.3 are elements of $\EllD(\Fp)$.  Since $\prod r_i=\#\clD=\#\EllD(\Fp)$, by Proposition~\ref{prop:CM}, we need only show that the outputs are distinct (and not equal to $j_0$).

To each vertex of the isogeny graph output by Algorithm~1.3 we associate a vector $\vec{x}\in X(\vec{\alpha})$ that identifies its position relative to $j_0$ in the sequence of paths computed.  The vector $(x_1,\ldots,x_k)$ identifies the vertex reached from $j_0$ via a path of length $x_k$ on the surface of the $\ell_k$-volcano, followed by a path of length $x_{k-1}$ on the surface of the $\ell_{k-1}$-volcano, and so forth.  We associate the zero vector to $j_0$.

Propositions \ref{prop:CM} and \ref{prop:volcano} imply that the vector $\vec{x}=(x_1,\ldots,x_k)$ corresponds to the action of some $\beta_{\vec{x}}\in \clD$.
For each integer $t_k$ in the interval $[0,r_k)$, the set of vectors of the form $(*,\ldots,*,t_k)$ corresponds to a coset of $G_{k-1}$ in the polycyclic series for $G=\clD$.  These cosets are distinct, regardless of the direction chosen by Algorithm~1.3 when starting its path on the $\ell_k$-volcano (note that $\alpha_k$ and $\alpha_k^{-1}$ have the same relative order $r_k$).  Proceeding inductively, 
for each choice of integers $t_i,t_{i+1},\ldots,t_k$ with $t_j\in[0,r_j)$ for $i\le j\le k$, the set of vectors of the form $(*,\ldots,*,t_i,t_{i+1},\ldots,t_k)$ corresponds to a distinct coset of $G_{i-1}$, regardless of the direction chosen by Algorithm~1.3 on the surface of the $\ell_i$-volcano.  Each coset of the cyclic group $G_0$ corresponds bijectively to a set of vectors of the form $(*,t_2,\ldots,t_k)$.  It follows that the $\beta_{\vec{x}}$ are all distinct. The action of $\clD$ is faithful, hence the outputs of Algorithm~1.3 are distinct.
\end{proof}

\subsection{Computing an optimal polycyclic presentation}\label{subsection:ComputePresentation}

Let $\vec{\gamma}=(\gamma_1,\ldots,\gamma_n)$ be a sequence of generators for a finite abelian group $G$, ordered by increasing cost (according to some cost function).  Then $\vec{\gamma}$ is a polycyclic sequence, and we may compute $r(\vec{\gamma})=(r_1,\ldots,r_n)$.  If we remove from $\vec{\gamma}$ each $\gamma_i$ for which $r_i=1$ and let $\vec{\alpha}=(\alpha_1,\ldots,\alpha_k)$ denote the remaining subsequence, then $\vec{\alpha}$ is a minimal polycyclic sequence for $G$.  We call $\vec{\alpha}$ the \emph{optimal} polycyclic sequence derived from~$\vec{\gamma}$.  It has $\alpha_1=\gamma_1$ with minimal cost, and for $i>1$ each $\alpha_i$ is the least-cost element not already contained in $G_{i-1}=\langle\alpha_1,\ldots,\alpha_{i-1}\rangle$.

We now give a generic algorithm to compute $r(\vec{\gamma})$ and a vector $s(\vec{\gamma})$ that encodes the power relations.  From $r(\vec{\gamma})$ and $s(\vec{\gamma})$, we can easily derive $\vec{\alpha}, r(\vec{\alpha})$, and $s(\vec{\alpha})$.  We define $s(\vec{\gamma})$ using a bijection $X(\vec{\gamma})\to\{z\in\Z: 0\le z < |G|\}$ given by:
\begin{equation}
Z(\vec{x}) = \sum_{1\le j\le n} N_jx_j,\qquad\text{where}\quad N_j=\prod_{1\le i<j}r_i.
\end{equation}
For each power relation $\gamma_i^{r_i}=\vec{\gamma}^\vec{x}$, we set $s_i=Z(\vec{x})$.  The formula
\begin{equation}
x_j=\left\lfloor s_i/N_j\right\rfloor \bmod r_j
\end{equation}
recovers the component $x_j$ of the vector $\vec{x}$ for which $s_i=Z(\vec{x})$.

\algstart{{\bf 2.2}}{Given $\vec{\gamma}=(\gamma_1,\ldots,\gamma_n)$ generating a finite abelian group $G$:}
\algitem
Let $T$ be an empty table and call $\textsc{TableInsert}(T,1_G)$ (so $T[0]=1_G$).
\algitem
For $i$ from $1$ to $n$:
\algitem
\qquad Set $\beta\leftarrow\gamma_i$, $r_i\leftarrow 1$, and $N\leftarrow\textsc{TableSize}(T)$.
\algitem
\qquad Until $s_i\leftarrow \textsc{TableLookup}(T, \beta)$ succeeds:
\algitem
\qquad\qquad For $j$ from 0 to $N-1$: $\textsc{TableInsert}(T, \beta\cdot T[j])$.
\algitem
\qquad\qquad Set $\beta\leftarrow\beta\gamma_i$ and $r_i\leftarrow r_i+1$.
\algitem
Output $r(\vec{\gamma})=(r_1,\ldots,r_n)$ and $s(\vec{\gamma})=(s_1,\ldots,s_n)$.
\algend

The table $T$ stores elements of $G$ in an array, placing each inserted element in the next available entry.
The function $\textsc{TableLookup}(T,\beta)$ returns an integer~$j$ for which $T[j]=\beta$ or fails if no such $j$ exists (when $j$ exists it is unique).
In practice lookups are supported by an auxiliary data structure, such as a hash table, maintained by $\textsc{TableInsert}$.
When group elements are uniquely identified, as with $\clD$, the cost of table operations is typically negligible.

\begin{proposition}\label{prop:pcp}
{\rm Algorithm}~$2.2$ is correct.  It uses $|G|$ non-trivial group operations, makes $|G|$ calls to \textsc{TableInsert}, and makes $\sum r_i$ calls to \textsc{TableLookup}.
\end{proposition}
\begin{proof}
We will prove inductively that $T[Z(\vec{x})]=\vec{\gamma}^\vec{x}$, and that each time the loop in Step 4 terminates, the values of $r_i$ and $s_i$ are correct and $T$ holds $G_i$.

When $i=1$ the algorithm computes $T[r_1]=\gamma_1^{r_1}T[0]$ for $r_1=1,2,\ldots$, until $\gamma_1^{r_1}=T[0]=1$, at which point $r_1=|\gamma_i|$, $s_1=0$, and $T$ holds $G_1$, as desired.

For $i>1$ we have $N=N_{i-1}$ and $T$ holds $G_{i-1}$ with $T[Z(\vec{x})]=\vec{\gamma}^\vec{x}$, by the inductive hypothesis.
For $r_i=1,2,\ldots$, if $\beta=\gamma_i^{r_i}$ is not in $T$, the algorithm computes $T[r_iN+j]=\gamma_i^{r_i}T[j]$, for $0\le j < N$, placing the coset $\gamma_i^{r_i}G_{i-1}$ in $T$.  When it finds $\gamma_i^{r_i}=T[s_i]$, the table $T$ contains all cosets of the form $\gamma_i^{r_i}G_{i-1}$ (since $G$ is abelian), hence $T$ holds $G_i$.  It follows that $r_i=|G_i:G_{i-1}|$ and $s_i$ is correct.

When the algorithm terminates, $T$ holds $G_n=G$, and every element of $G$ is inserted exactly once.  A group operation is performed for each call to \textsc{TableInsert}, but in each execution of Step 5 the first of these is trivial, and we instead count the non-trivial group operation in Step 6.  The number of calls to $\textsc{TableLookup}$ is clearly the sum of the $r_i$, which completes the proof.
\end{proof}

The complexity of Algorithm~2.2 is largely independent of $\vec{\gamma}$.  When $\vec{\gamma}$ contains every element of $G$, Algorithm~2.2 is essentially optimal.  However, if $\vec{\gamma}$ has size $n=o(|G|^{1/2})$, we can do asymptotically better with an $O(n|G|^{1/2})$ algorithm.
This is achieved by computing a basis $\vec{\alpha}$ for $G$ via a generic algorithm (as in  \cite{Buchmann:GroupStructure,Sutherland:Thesis,Sutherland:AbelianpGroups,Teske:GroupStructure}), and then determining the representation of each $\gamma_i=\vec{\alpha}^\vec{x}$ in this basis using a vector discrete logarithm algorithm (such as \cite[Alg.~9.3]{Sutherland:Thesis}).  It is then straightforward to compute $|G_i|$ for each $i$ and from this obtain $r_i=|G_i:G_{i-1}|$.  The power relations can then be computed using discrete logarithms with respect to $\vec{\gamma}$.  In the specific case $G=\clD$, one may go further and use a non-generic algorithm to compute a basis $\vec{\alpha}$ in subexponential time (under the ERH) \cite{Hafner:SubexponentialClassGroups}, and apply a vector form of the discrete logarithm algorithm in \cite{Vollmer:DiscreteLogClassGroup}.

\subsection{Application to $\boldsymbol{\clD}$}
For the practical range of $D$, the group $G=\clD$ is relatively small (typically $|G|< 10^8$), and the constant factors make Algorithm~2.2 faster than alternative approaches; even in the largest examples of Section~8 it takes only a few seconds.  Asymptotically, Algorithm~2.2 uses $O(|D|^{1/2+\epsilon})$ time and $O(|D|^{1/2}\log^2|D|)$ space to compute an optimal polycyclic sequence for $\clD$.  In fact, under the GRH, we can compute a separate polycyclic sequence for every $v(p)$ arising among the primes $p\in S$ that are selected by Algorithm~2.1 (Section~\ref{subsection:PickPrimes}) within the same complexity bound, by Lemma~\ref{lemma:pvmax} (Section~\ref{section:Complexity}).

We uniquely represent elements of $\clD$ with primitive, reduced, binary quadratic forms $ax^2+bxy+cy^2$, where $a$ corresponds to the norm of a reduced ideal representing its class.  For the sequence $\vec{\gamma}$ we use forms with $a=\ell$ prime, constructed as in \cite[Alg.~3.3]{Buchmann:BinaryQuadraticForms}.  Under the ERH, restricting to $\ell\le 6\log^2|D|$ yields a sequence of generators for $\clD$, by \cite{Bach:ERHBounds}.  To obtain an unconditional result, we precompute $h(D)$ and extend $\vec{\gamma}$ dynamically until Algorithm~2.2 reaches $N=h(D)$.

We initially order the elements $\gamma_i$ of $\vec{\gamma}$ by their norm $\ell_i$, assuming that this reflects the cost of using the action of $\gamma_i$ to enumerate $\EllD(\Fp)$ via Algorithm~1.3 (Section 4.2).  However, for those $\ell_i$ that divide $v(p)$ we may wish to adjust the relative position of $\gamma_i$, since walking the surface of an $\ell_i$-volcano with nonzero depth increases the average cost per step.  We use Proposition~\ref{prop:WalkSurface} to estimate this cost, which may or may not cause us to change the position of $\gamma_i$ in~$\vec{\gamma}$.  In practice just a few (perhaps one) distinct orderings suffice to optimally address every $v(p)$.

Note that we need not consider the relative orders $r_i$ when ordering $\vec{\gamma}$.  If $i$ is less than $j$, then Algorithm~1.3 always takes at least as many steps using $\ell_i$ as it does using $\ell_j$.  Indeed, the running time of Algorithm~1.3 is typically determined by the choice of $\alpha_1$: at least half of the steps will be taken on the surface of an $\ell_1$-volcano, and if $\inkron{D}{\ell_1}=1$, almost all of them will (heuristically).

\subsection{Why not use a basis?}\label{section:PCPvsBAsis}
Using a basis to enumerate $\EllD(\Fp)$ is rarely optimal, and in the worst case it can be a very poor choice.  The ERH does imply that $\clD$ is generated by the classes of ideals with prime norm $\ell\le 6\log^2|D|$, but this set of generators need not contain a basis.  As a typical counterexample, consider
$$D_1=-10007\cdot 10009\cdot 10037,$$
the product of the first three primes greater than $10000$.  The class group has order $h(D_1)=2^2\cdot 44029$, where 44029 is prime, and its 2-Sylow subgroup $H$ is isomorphic to $\Z/2\Z\times\Z/2\Z$.  Every basis for $\cl(D_1)$ must contain a non-trivial element of $H$, and these classes have reduced representatives with norms 10007, 10009, and 10037, all of which are greater than $6\log^2|D_1|\approx 4583$.

By comparison, Algorithm~2.2 computes an optimal polycyclic sequence for $\cl(D_1)$ with $\ell_1=5$ and $\ell_2=37$ (and relative orders $r_1=88058$ and $r_2=2$).

\section{Chinese Remaindering}\label{section:CRT}

As described in Section 2, for each coefficient $c$ of the Hilbert class polynomial we may derive the value of $c\bmod P$ (for any positive integer $P$) from the values $c_i\equiv c \bmod p_i$ appearing in $H_D\bmod p_i$ (for $p_i\in S$), using an explicit form of the Chinese Remainder Theorem (CRT). We apply
\begin{equation}
c \equiv \sum c_ia_iM_i-rM \bmod P,\tag{\ref{equation:explicitCRT}}
\end{equation}
where $M=\prod p_i$, $M_i=M/p_i$, $a_i=M_i^{-1}\bmod p_i$, and $r$ is the closest integer to $s=\sum c_ia_i/p_i$.
Recall that $S\subset\PD$ is chosen so that $M>4B$, where $B$ bounds the coefficients of $H_D$, via Lemma~\ref{lemma:Bbound}.  It suffices to approximate each term in the sum $s$ to within $1/(4n)$, where $n=\#S$.  If $\pmax$ denotes the largest $p_i$, we need $O\bigl(\log(n(\pmax+\log n))\bigr) = O(\log\pmax)$ bits of precision to compute $r$.

To minimize the space required, we accumulate $C=\sum c_ia_iM_i\bmod P$ and an approximation of $s$ as the $c_i$ are computed.  This uses $O(\log P+\log \pmax)$ space per coefficient.  We have $h(D)$ coefficients to compute, yielding
\begin{equation}\label{equation:spacebound}
O\bigl(h(D)(\log P + \log\pmax)\bigr)
\end{equation}
as our desired space bound.

To achieve this goal without increasing the time complexity of our algorithm, we consider two cases: one in which $P$ is small, which we take to mean
\begin{equation}
\log P\le \mu \log^3 |D|,
\end{equation}
for some absolute constant $\mu$, and another in which $P$ is large (not small).  The former case is typical when applying the CM method; $P$ may be a cryptographic-size prime, but it is not unreasonably large.  The latter case most often arises when we actually want to compute $H_D$ over $\Z$.  When $P\ge M$ there is no need to use the explicit CRT and we apply a standard CRT computation.  To treat the intermediate case, where $P$ is large but smaller than $M$, we use a hybrid approach.

The optimal choice of $\mu$ depends on the relative cost of performing $h(D)$ multiplications modulo $P$ versus the cost of computing $H_D \bmod p_i$; we want the former to be small compared to the latter.  In practice, the constant factors allow us to make $\mu$ quite large and the intermediate case rarely arises.

\subsection{Fast Chinese remaindering in linear space}

Standard algorithms for fast Chinese remaindering can be found in \cite[\S 10.3]{Gathen:ComputerAlgebra}.  We apply similar techniques, but use a time/space trade-off to achieve the space bound in (\ref{equation:spacebound}).  These computations involve a \emph{product tree} built from coprime moduli.  In our setting these are the primes $p_i\in S$, which we index here as $p_0,\ldots,p_{n-1}$.

We define a product tree as a leveled binary tree in which each vertex at level $k$ is either a leaf or the product of its two children at level $k+1$ (we require levels to have an even number of vertices and add a leaf to levels that need one).  It is convenient to label the vertices by bit-strings of length $k$, where the root at level 0 is labeled by the empty string and all other vertices are uniquely labeled by appending the string ``0" or ``1" to the label of their parent.

Let $d=\lfloor\lg(n-1)\rfloor+1$ be the number of bits in the positive integer $n-1$.  For integers $i$ from 0 to $n-1$, we let $b(i)\in\{0,1\}^d$ denote the bit-string corresponding to the binary representation of $i$.
The products $m_x$ are defined by placing the moduli in leaves as $m_{b(i)}=p_i$, setting $m_x=1$ for all other leaves, and defining $m_x=m_{x0}m_{x1}$ for all internal vertices.

The modular complements $\overline{m}_x=m/m_x \bmod m_x$ are then obtained by setting $\overline{m}_0=m_1\bmod m_0$ and $\overline{m}_1=m_0\bmod m_1$, and defining
$$\overline{m}_{x0}=\overline{m}_xm_{x1}\bmod m_{x0}\qquad\text{and}\qquad\overline{m}_{x1}=\overline{m}_xm_{x0}\bmod m_{x1}.$$
In terms of $M_i=M/p_i$, we then have $m=M$ and $\overline{m}_{b(i)}=M_i\bmod p_i$.
\smallskip

Let $I_k$ denote the labels at level $k$, for $1\le k\le d$ (and otherwise $I_k$ is empty).  One way to compute $\overline{m}_d$ is as follows:
\vspace{3pt}
\begin{enumerate}
\item
For $k$ from $d$ to 1, compute $m_x$ for $x\in I_k$.
\item
For $k$ from 1 to $d$, compute $\overline{m}_x$ for $x\in I_k$.
\end{enumerate}
\vspace{3pt}
This uses $O(\textsf{M}(\log{M})\log{n})$ time and $O(\log{M}\log{n})$ space.  Alternatively:
\vspace{3pt}
\begin{enumerate}
\item
For $k$ from $1$ to d:
\item
\quad For $j$ from $d$ to $k$, compute $m_x$ for $x\in I_j$ (discard $m_y$ for $y\in I_{j+1}$).
\item
\quad Compute $\overline{m}_x$ for $x\in I_k$ (discard $m_y$ for $y\in I_k$ and $\overline{m}_z$ for $z\in I_{k-1}$). 
\end{enumerate}
\vspace{3pt}
This uses $O(\textsf{M}(\log{M})\log^2{n})$ time and $O(\log{M})$ space.  In general, storing $\lceil\log^\omega{n}\rceil$ levels uses $O(\textsf{M}(\log{M})\log^{2-\omega}{n})$ time and $O(\log{M}\log^\omega{n})$ space, for $0\le\omega\le1$.
\smallskip


\subsection{Applying the explicit CRT when $P$ is small}\label{subsection:SmallP}
Assume $\log P\le \mu\log^3|D|$.  We index the set $S\subset\PD$ as $S=\{p_0,\ldots,p_{n-1}\}$ and let $M=\prod p_i$ and $M_i=M/p_i$.  As above, we define products $m_x$ and modular complements $\overline{m}_x=m/m_x\bmod m_x$, and similarly define modular complements $\overline{m}_x'=m/m_x \bmod P$.

\algstart{{\bf 2.3 (precompute)}}{Given $S=\{p_0,\ldots,p_{n-1}\}$ and $P$:}
\algitem
Compute $\overline{m}_x$ and $\overline{m}_x'$.  Save $M\bmod P$.
\algitem
Use $\overline{m}_{b(i)}\equiv M_i\bmod p_i$ to set $a_i\leftarrow M_i^{-1}\bmod p_i$.
\algitem
Use $\overline{m}_{b(i)}'\equiv M_i\bmod P$ to set $d_i\leftarrow a_iM_i\bmod P$.
\algitem
Set $C_j\leftarrow 0$ and $s_j\leftarrow 0$ for $j$ from 0 to $h(D)$.
\algend
Using the time/space trade-off described above, Algorithm~2.3 has a running time of
$O(\textsf{M}(\log M)\log^2 n)$, using $O(\log M+n\log P)$ space.

We now set $\delta=\lceil\lg{n}\rceil+2$, which determines the precision of the integer $s_j\approx 2^{\delta}r$ we use to approximate the rational number $r$ in (\ref{equation:explicitCRT}).
\smallskip

\algstart{{\bf 2.4 (update)}}{Given $H_D\bmod p_i$ with coefficients $c_j$:}
\algitem 
For $j$ from 0 to $h(D)$:
\algitem
\qquad Set $C_j\leftarrow C_j+c_jd_i \bmod P$.
\algitem
\qquad Set $s_j\leftarrow s_j + \lfloor 2^\delta c_ja_i/p_i\rfloor$.
\algend
The total running time of Algorithm~2.4 over all $p_i\in S$ may be bounded by
\begin{equation}\label{equation:CRTupdate}
O\bigl(nh(D)\textsf{M}(\log P)+h(D)\textsf{M}(\log M+n\log n)\bigr).
\end{equation}
Typically the first term dominates, and it is here that we need $\log P=O(\log^3|D|)$.  The space complexity is $O(h(D)(\log{P}+\log\pmax+\log n))$.

\algstart{{\bf 2.5 (postcompute)}}{After computing $H_D\bmod p_i$ for all $p_i\in S$:}
\algitem 
For $j$ from 0 to $h(D)$:
\algitem
\qquad
Set $C_j\leftarrow C_j-\lfloor3/4+ 2^{-\delta}s_j\rfloor M\bmod P$.
\algitem Output $H_D\bmod P$ with coefficients $C_j$.
\algend
Algorithm~2.5 uses $O(h(D)\textsf{M}(\log P))$ time and $O(h(D)\log P)$ space.  The formulas used by Algorithms 2.4 and 2.5 are taken from \cite[Thm. 2.2]{Bernstein:ModularExponentiation} (also see \cite{Bernstein:ExplicitCRT}).

\subsection{Applying the CRT when $P$ is large}\label{subsection:LargeP}

When $P$ is larger than $M$, we simply compute $H_D\in\Z[X]$ using a standard application of the CRT.  That is, we compute $H_D\bmod p_i$ for $p_i\in S$, and then apply
\begin{equation}\tag{\ref{equation:CRT}}
c\equiv\sum c_ia_iM_i\bmod M
\end{equation}
to compute each coefficient of $H_D$ using fast Chinese remaindering \cite[\S 10.3]{Gathen:ComputerAlgebra}.  Since $M>2B$, this determines $H_D\in\Z[X]$.  Its coefficients lie in the interval $(-P/2,P/2)$, so we regard this as effectively computing $H_D\bmod P$.  The total time spent applying the CRT is then $O(h(D)\textsf{M}(\log M)\log n)$, and the space needed to compute $(\ref{equation:CRT})$ is $O(\log M\log n)$, which is easily smaller than the $O(h(D)\log M)$ bound on the size of $H_D$ (so no time/space trade-off is required).

When $P$ is smaller than $M$ but $\log P > \mu \log^3|D|$, we combine the two CRT approaches.  We group the primes $p_0,\ldots,p_{n-1}$ into products $q_0,\ldots,q_{k-1}$ so that $\log q_j\approx \log P$ (or $q_j>\log P$ is prime).  We  compute $H_D\bmod q_j$ by applying the usual CRT to the coefficients of $H_D\bmod p_i$, after processing all the $p_i$ dividing~$q_j$.  If $q_j$ is prime no work is involved, and otherwise this takes $O(\textsf{M}(\log P)\log{n})$ time per coefficient.  We then apply the explicit CRT to the coefficients of $H_D\bmod q_j$, as in Section \ref{subsection:SmallP}, discarding the coefficients of $H_D\bmod q_j$ after they have been processed by Algorithm~2.4.  This hybrid approach has a time complexity of
\begin{equation}\label{equation:CRThybrid}
O(h(D)(\log M/\log P)\M(\log P)\log n) = O(h(D)\M(\log M)\log n),
\end{equation}
and uses $O\bigl(h(D)(\log{P}+\log \pmax)\bigr)$ space.

\section{Complexity Analysis}\label{section:Complexity}
We now analyze the complexity of Algorithms 1 and 2, proving Theorem~1 through a series of lemmas.
To do so, we apply various number-theoretic bounds that depend on some instance of the extended or generalized Riemann hypothesis.
We use the generic label ``GRH" to identify all statements that depend (directly or indirectly) on one or more of these hypotheses.
As noted in the introduction, the GRH is used only to obtain complexity bounds, the outputs of Algorithms 1 and 2 are unconditionally correct.

Let $\M(n)$ denote the cost of multiplication, as defined in \cite[Ch.~8]{Gathen:ComputerAlgebra}.  We have
\begin{equation}\label{equation:multbound}
\M(n)=O(n\log{n}\llog n),
\end{equation}
by \cite{Schonhage:Multiplication}, where $\llog(n)$ denotes $\log\log n$ (and we use $\lllog(n)$ to denote $\log\log\log n$).
We focus here on asymptotic results and apply (\ref{equation:multbound}) throughout, noting that the larger computations
in Section~\ref{section:Performance} make extensive use of algorithms that realize this bound.  See Section ~\ref{subsection:heuristics} for a practical discussion of $\M(n)$.

Let us recall some key parameters.  For a discriminant $D<-4$, we define
\begin{equation}
\PD = \{p>3\medspace {\rm prime}: 4p=t^2-v^2D \medspace\text{for some} \medspace t, v\in \Z_{>0}\},\tag{\ref{equation:PD}}
\end{equation}
where $t=t(p)$ and $v=v(p)$ are uniquely determined by $p$.  We select a subset
$$S\subseteq S_z=\{p\in\PD:p/H(-v(p)^2D) \le z\},$$
that satisfies $\prod_{p\in S}p > 4B$, where $B$ bounds the absolute values of the coefficients of $H_D$.
We also utilize prime norms $\ell_1,\ldots,\ell_k$ arising in a polycyclic presentation of $\cl(D)$ that is derived from a set of generators.
\medskip

\noindent
(\textbf{GRH}) For convenient reference, we note the following bounds:

\smallskip
\renewcommand\theenumi{\roman{enumi}}
\renewcommand\labelenumi{(\theenumi)}
\begin{enumerate}
\item
$h = {h(D)}=O(|D|^{1/2}\llog|D|)$ (see \cite{Littlewood:ClassNumber}).
\vspace{6pt}
\item
${b}=\lg{B}+2=O(|D|^{1/2}\log|D| \llog|D|)$ (Lemma~\ref{lemma:Bbound}).
\vspace{6pt}
\item
${n}=\#S=O(|D|^{1/2}\llog|D|)$ (follows from (ii)).
\vspace{6pt}
\item
${\lmax}=\max\{\ell_1,\ldots,\ell_k\}=O(\log^2|D|)$ (see \cite{Bach:ERHBounds}).
\vspace{6pt}
\item
${z}=O(|D|^{1/2}\log^3|D|\llog|D|)$ (Lemma~\ref{lemma:rbound}).
\vspace{6pt}
\item
${\pmax}=\max{S}=O(|D|\log^6|D|\llog^8|D|)$ (Lemma \ref{lemma:pvmax}).
\vspace{6pt}
\item
${\vmax}=\max\{v(p):p\in S\}=O(\log^3|D|\llog^4|D|)$ (Lemma \ref{lemma:pvmax}).
\end{enumerate}
\medskip

The first three parameters have unconditional bounds that are only slightly larger (see \cite[\S 5.1]{Belding:HilbertClassPolynomial}), but the last four depend critically on either the ERH or GRH.  Heuristic bounds are discussed in Section \ref{subsection:heuristics}.

To prove (v) we use an effective form of the Chebotarev density theorem \cite{Lagarias:Chebotarev}.
Recall that $\PD$ is the set of primes (greater than 3) that split completely in the ring class field $K_\O$ of $\O$.
For a positive real number $x$, let $\pi_1(x,K_\O/\Q)$ count the primes $p\le x$ that split completely in $K_\O$.
Equivalently, $\pi_1(x,K_\O/\Q)$ counts primes whose image in $\Gal(K_\O/\Q)$ under the Artin map is the identity element \cite[Cor.~5.21]{Cox:ComplexMultiplication}.  Applying Theorem~1.1 of \cite{Lagarias:Chebotarev} then yields
\begin{equation}\label{equation:Chebotarev}
\left\vert\pi_1(x,K_\O/\Q)-\frac{{\rm Li}(x)}{2h(D)}\right\vert\le
c_1\left(\frac{x^{1/2}\log\left(|D|^{h(D)}x^{2h(D)}\right)}{2h(D)} + \log(|D|^{h(D)})\right),
\end{equation}
as in \cite[Eq.~3]{Belding:HilbertClassPolynomial}, where the constant $c_1$ is effectively computable.

\begin{lemma}[\textbf{GRH}]\label{lemma:rbound}
For any real constant $c_3$ there is an effectively computable constant~$c_2$ such that
$z\ge c_2h(D)\log^3|D|$ implies $\#S_z \ge c_3h(D)\log^3|D|$.
\end{lemma}
\begin{proof}
Let $h=h(D)$.  We apply (\ref{equation:Chebotarev}) to $x=c_0h^2\log^4|D|$, with $c_0$ to be determined.  We assume $D<-4$ and $\log c_0 \ge 2$, which implies $\log x < 4\log c_0\log|D|$ (using $h<|D|$ and $\log|D|<|D|^{1/2}$), and $\Li(x) > x/\log x$, for all $x\ge 1$.  Negating the expression within the absolute value, we obtain from (\ref{equation:Chebotarev}) the inequality
$$
\pi_1(x,K_\O/\Q)\ge \left(\frac{c_0}{8\log c_0} - 5c_1\sqrt{c_0}\log c_0\right)h\log^3|D|.
$$
Thus given any constant $c_4$ we may effectively determine $c_0\ge e^2$ (using $c_1$) so that
$$
\pi_1(x,K_\O/\Q) \sge c_4h\log^3|D|.
$$
For the set $R_x$ of primes in $\PD$ bounded by $x$, we have $\#R_x=\pi_1(x,K_\O/\Q)-2$.

Let $v_0$ be the least integer such that at least half the primes in $R_x$ have $v(p)\le v_0$.
There are $v_0$ positive integers less than or equal to $v_0$, and any particular value $v(p)\le v_0$ can arise for
at most $2\sqrt{x}$ primes $p\in R_x$, since $t(p) < 2\sqrt{p}\le 2\sqrt{x}$.
Therefore $2v_0\sqrt{x} \ge \#R_x /2$, and this implies
$$
2v_0\sqrt{c_0}h\log^2|D|\sge (c_4h\log^3|D|-2)/2 \sgt (c_4/2-1)h\log^3|D|.
$$
We thus obtain $v_0 > c_5\log|D|$, where $c_5 = (c_4/2-1)/\sqrt{4c_0}$, and assume $c_4 > 2$.

For primes $p\in R_x$ with $v(p)\ge v_0$, the lower bound in Lemma~\ref{lemma:vbound} implies
$$
\frac{p}{H(-v(p)^2D)}\sle \frac{p}{v(p)H(-D)} \sle \frac{x}{c_5h\log|D|}\seq(c_0/c_5)h\log^3|D|.
$$
If $z\ge c_2h\log^3|D|$, with $c_2=c_0/c_5$, then $S_z$ contains at least half the primes in~$R_x$.
Setting $c_4=\max\{2c_3+2,3\}$ determines $c_0$, $c_5$, and $c_2$, and completes the proof.
\end{proof}

The primes $p\in S_z$ are enumerated by Algorithm~2.1 (Section~\ref{subsection:PickPrimes}), which gradually increases $z$ until $\sum_{p\in S_z} \lg p > 2b$, where $b=\lg B + 2$.

\begin{lemma}[\textbf{GRH}]\label{lemma:pvmax}
When {\rm Algorithm}~$2.1$ terminates, for every prime $p\in S_z$ we have the bounds $p=O(|D|\log^6|D|\llog^8|D|)$ and $v(p)=O(\log^3|D|\llog^4|D|)$.
\end{lemma}
\begin{proof}
Let $D=u^2D_K$, where $u$ is the conductor of $D$.  The upper bound in Lemma~\ref{lemma:vbound}, together with the bound (i) on $h(D)$, implies that for a suitable constant $c_2$ and sufficiently large $|D|$, the bound
$$
H(-v^2D) \sle 12uvH(-D_K)\llog^2(uv+4)) \sle c_2v|D|^{1/2}\llog|D|\llog^2(v|D|)
$$
holds for all positive integers $v$.

Lemma~\ref{lemma:rbound}, together with bounds (i) and (ii), implies that Algorithm~2.1 achieves
$\sum_{p\in S_z} \lg p > 2b$ with $z=O(h(D)\log^3|D|)=O(|D|^{1/2}\log^3|D|\llog|D|)$.  Thus
for a suitable constant $c_3$ and sufficiently large $|D|$, the bound
\begin{equation}\label{equation:pSzbound}
p \sle zH(-v(p)^2D) \sle c_3v(p)|D|\log^3|D|\llog^2|D|\llog^2(v(p)|D|)
\end{equation}
holds for all $p\in S_z$.  We also have $v(p)\le 2\sqrt{p/|D|}$, since $4p=t(p)^2-v(p)^2D$.
Applying this inequality to (\ref{equation:pSzbound}) yields $p=O(|D|\log^6|D|\llog^8|D|)$, which then implies $v=O(\log^3|D|\llog^4|D|)$.
\end{proof}

We could obtain tighter bounds on $\pmax$ and $\vmax$ by modifying Algorithm~2.1 to only consider primes in $R_x\cap S_z$, but there is no reason to do so.  Larger primes will be selected for $S$ only when they improve the performance.

To achieve the space bound of Theorem~1, we assume a time/space trade-off is made in the implementation of Algorithm~2.1.
We control the space used to find the primes in $S_z$, by sieving within a suitably narrow window.  This increases
the running time by a negligible poly-logarithmic factor.

\begin{lemma}[\textbf{GRH}]\label{lemma:sieve}
The expected running time of {\rm Algorithm}~$2.1$ is $O(|D|^{1/2+\epsilon})$, using $O(|D|^{1/2}\log|D|\llog|D|)$ space.
\end{lemma}
\begin{proof}
When computing $S_z$, it suffices to consider $v$ up to an $O(\log^{3+\epsilon}|D|)$ bound, by Lemma \ref{lemma:pvmax} above.  For each $v$ we sieve the polynomial $f(t)=t^2-v^2D$ to find $f(t)=4p$ with $p$ prime.  The bound on $p$ implies that we need only sieve to an $L=O(|D|^{1/2}\log^{3+\epsilon}|D|)$ bound on $t$.  We may enumerate the primes up to $L$ in $O(L\llog{L})$ time using $O(\sqrt{L}\log L)=O(|D|^{1/4+\epsilon})$ space (we sieve with primes up to $\sqrt{L}$ to identify primes up to $L$ using a window of size $\sqrt{L}$).

For each of the $\pi(L)$ primes $\ell\le L$, we compute a square root of $-v^2D$ modulo~$\ell$ probabilistically, in expected time $O(\M(\log \ell)\log \ell)$, and use it to sieve $f(t)$.  Here we sieve using a window of size $O(|D|^{1/2}\log|D|\llog|D|)$, recomputing each square root $O(\log^{2+\epsilon}|D|)$ times in order to achieve the space bound.

For each $v$, the total cost of computing square roots is
$O(\pi(L)\log^{4+\epsilon}|D|)$, which dominates the cost of sieving.  Applying $\pi(L)=O(L/\log{L})$ and summing over $v$ yields $O(|D|^{1/2}\log^{9+\epsilon}|D|)$, which dominates the time to select $S\subset S_z$.

To stay within the space bound, if we find that increasing $z$ in Step 3 by a factor of $1+\delta$ causes $S_z$ to be too large (say, greater than $4b$ bits), we backtrack and instead increase $z$ by a factor of $1+\delta/2$ and set $\delta\leftarrow\delta/2$.  We increase $z$ a total of $O(\log|D|)$ times (including all backtracking), and the lemma follows.
\end{proof}
In practice we don't actually need to make the time/space tradeoff described in the proof above.
Heuristically we expect $\pmax = O(|D|\log^{1+\epsilon}|D|)$, and in this case all the primes in $S_z$ can be found in a single pass with $L=O(|D|^{1/2}\log^{1/2+\epsilon}|D|)$.

We now show that all the precomputation steps in Algorithm 2 take negligible time and achieve the desired space bound.
This includes selecting primes (Algorithm~2.1 in Section \ref{subsection:PickPrimes}), computing polycyclic presentations (Algorithm~2.2 in Section \ref{subsection:ComputePresentation}), and CRT precomputation (Algorithm 2.3 in Section \ref{subsection:SmallP}).

\begin{lemma}[\textbf{GRH}]\label{lemma:precompute}
Steps $1$, $2$, and $3$ of {\rm Algorithm} $2$ take $O(|D|^{1/2+\epsilon})$ expected time and use
$O(|D|^{1/2}(\log|D|+\log P)\llog|D|)$ space.
\end{lemma}
\begin{proof}
The complexity of Step 1 is addressed by Lemma~\ref{lemma:sieve} above.
By Proposition~\ref{prop:pcp}, Step 2 performs $h(D)$ operations in $\clD$, each taking $O(\log^2|D|)$ time \cite{Biehl:FormReductionComplexity}.  Even if we compute a different presentation for every $v\le \vmax$, the total time is $O(|D|^{1/2+\epsilon})$.
The table used by Algorithm~2.2 stores $h(D)=O(|D|^{1/2}\llog|D|)$ group elements, by bound (i), requiring $O(|D|^{1/2}\log|D|\llog|D|)$ space.  

As described in Section \ref{subsection:SmallP}, when $\log P\le \mu\log^3|D|$ the complexity of Algorithm~2.3 is
$O(\textsf{M}(\log M)\log^2 n)$ time and $O(\log M+n\log P)$ space, and we have
$$\log M=\sum_{p\in S}\log p\le n\log \pmax = O(|D|^{1/2}\log|D|\llog|D|),$$
according to bounds (iii) and (vi) above.
As discussed in Section \ref{subsection:LargeP}, the same time and space bounds for precomputation apply when $\log P > \mu\log^3|D|$.
\end{proof}

We next consider \textsc{TestCurveOrder} (Section \ref{subsection:TestingCurves}), which is used by Algorithm~1.1 to find a curve in $\Ellt(\Fp)$.  We assume \cite[Alg.~7.4]{Sutherland:Thesis} is used to implement the algorithm \textsc{FastOrder} which is called by \textsc{TestCurveOrder}.

\begin{lemma}\label{lemma:TestCurve}
\textsc{TestCurveOrder} runs in expected time $O(\log^2{p}\llog^2{p})$.
\end{lemma}
\begin{proof}
For $s=0,1$ the integer $m_s$ computed by \textsc{TestCurveOrder} is the lcm of the orders of random points in $E_s(\Fp)$.  By \cite[Thm.~8.1]{Sutherland:Thesis} we expect $O(1)$ points yield $m_s=\lambda(E_s(\Fp))$, the group exponent of $E_s(\Fp)$.
For $p>11$, Theorem~2 and Table~1 of \cite{CremonaSutherland:MestreSchoof} then imply $\mathcal{N}\subseteq\{N_0,N_1\}$, forcing termination.
We thus expect to execute each step $O(1)$ times.  We now bound the cost of Steps 2-5:
\renewcommand\theenumi{\arabic{enumi}}
\renewcommand\labelenumi{\theenumi.}
\begin{enumerate}
\item[2.]
The non-residue used to compute $\tilde{E}$ can be probabilistically obtained using an expected $O(\log p)$ operations in $\Fp$, via Euler's criterion.
\item[3.]
With $E_s$ in the form $y^2=f(x)$, we obtain a random point $(x,y)$ by computing the square-root of $f(x)$ for random $x\in\Fp$, using an expected $O(\log p)$ operations in $\Fp$ to compute square roots (probabilistically).
\item[4.]
Computing $Q=m_sP$ uses $O(\log p)$ group operations in $E_s(\Fp)$.  The factorization of $N_s/m_s$ is obtained by maintaining $m_s$ in factored form.  Implementing \textsc{FastOrder} via \cite[Alg.~7.4]{Sutherland:Thesis} uses
$O(\log p\llog p/\lllog p)$ group operations on $E_s(\Fp)$, by \cite[Prop.~7.3]{Sutherland:Thesis}.
\item[5.]
The intersection of two arithmetic sequences can computed with the extended Euclidean algorithm in time $O(\log^2 p)$, by \cite[Thm.~3.13]{Gathen:ComputerAlgebra}.
\end{enumerate}
Step~4 dominates.  The group operation in $E_s(\Fp)$ uses $O(1)$ operations in $\Fp$, each with bit complexity $O(\M(\log p))$, and this yields the bound of the lemma.
\end{proof}

We are now ready to bound the complexity of Algorithm~1 (Section \ref{section:Overview}), which computes $H_D\bmod p$ using Algorithm~1.1 (Section \ref{subsection:TestingCurves}), Algorithm~1.2 (Section \ref{subsection:FindEllD}), and Algorithm~1.3 (Section \ref{subsection:EnumEllD}).

\begin{lemma}[\textbf{GRH}]\label{lemma:alg1}
For $p\in S$, {\rm Algorithm} $1$ computes $H_D\bmod p$ with an expected running time of $O(|D|^{1/2}\log^5|D|\llog^3|D|)$, using $O(|D|^{1/2}\log|D|\llog|D|)$ space.
\end{lemma}
\begin{proof}
Ignoring the benefit of any torsion constraints, Algorithm~1.1 expects to sample $p/H(-v^2D)\le z$ random curves over $\Fp$ to find $j\in \Ellt(\Fp)$.  The cost of testing a curve is $O(\log^2{p}\llog^2{p})$, by Lemma~\ref{lemma:TestCurve}, and this bound dominates the cost of any filters applied prior to calling \textsc{TestCurveOrder}.

Applying bound (v) on $z$ and bound (vi) on $\pmax$ yields an overall bound of
\begin{equation}\label{equation:Alg11bound}
O(|D|^{1/2}\log^5|D|\llog^3|D|)
\end{equation}
on the expected running time of Algorithm~1.1, and it uses negligible space.

Algorithm~1.2 finds $j\in\EllD(\Fp)$ in polynomial time if the conductor of $D$ is small, and otherwise its complexity is bounded by the $O(p^{1/3})=O(|D|^{1/3+\epsilon})$ complexity of Kohel's algorithm (under GRH).  In either case it is negligible.

As shown in \cite{Bach:ERHBounds}, the ERH yields an $O(\log^2|D|)$ bound on the prime norms needed to generate $\clD$, even if we exclude norms dividing $v$ (at most $O(\llog|D|)$ primes).  It follows that every optimal polycyclic presentation used by Algorithm~1.2 has norms bounded by $\lmax=O(\log^2|D|)$.  To bound the running time of Algorithm~1.3 we assume $\ell_i\nmid v$, since we use $\ell_i|v$ only when it improves performance.

The time to precompute each $\Phi_{\ell_i}$ is $O(\ell_i^{3+\epsilon})=O(\log^{6+\epsilon}|D|)$, by \cite{Enge:ModularPolynomials}, and at most $O(\log|D|)$ are needed.  These costs are negligible relative to the desired bound, as is the cost of reducing each $\Phi_{\ell_i}$ modulo $p$.  Applying the bound on $\lmax$ and bound (vi) on $\pmax$, each step taken by Algorithm~1.3 on an $\ell_i$-isogeny cycle uses $O(\log^4|D|)$ operations in $\Fp$, by (\ref{equation:isogenycost}).  A total of $h$ steps are required, and the bounds (i) on $h$ and (vi) on $p$ yield a bit complexity of $O(|D|^{1/2}\log^5|D|\llog^{2+\epsilon}|D|)$ for Algorithm~1.3, using $O(h\lg p)=O(|D|^{1/2}\log|D|\llog|D|)$ space.

Step 4 of Algorithm~1 computes $\prod(X-j)$ over $j\in\EllD(\Fp)$ via a product tree, using $O(\M(h)\log h)$ operations in $\Fp$ and space for two levels of the tree.  Applying bound (i), this uses $O(|D|\log^{3+\epsilon}|D|)$ time and $O(|D|^{1/2}\log|D|\llog|D)$ space.
\end{proof}

The time bound in Lemma~\ref{lemma:TestCurve} may be improved to $O(|D|^{1/2}\log^5|D|\llog^2|D|)$ by arguing that a random point on a random elliptic curve over $\Fp$ has order greater than $4\sqrt{p}$ with probability $1-O(1/\log p)$.

\begin{thm1}[\textbf{GRH}]
{\rm Algorithm 2} computes $H_D \bmod P$ in $O(|D|\log^{5}|D|\llog^4|D|)$ expected time, using $O(|D|^{1/2}(\log|D|+\log P)\llog|D|)$ space.
\end{thm1}
\begin{proof}
Lemma~\ref{lemma:precompute} bounds the cost of Steps 1--3.  As previously noted, if we have $P> M=\prod_{p\in S}p$, we set $P=M$ and compute $H_D$ over $\Z$.

Algorithm~1 is called for each $p\in S$, of which there are $n=O(|D|^{1/2}\llog|D|)$, by bound (iii).  Applying Lemma~\ref{lemma:alg1}, Algorithm~1 computes $H_D\bmod p$ for all $p\in S$ within the time and space bounds stated in the theorem.

Recalling (\ref{equation:CRTupdate}) from Section \ref{subsection:SmallP}, for $\log P\le\mu \log^3|D|$ the total cost of updating the CRT sums via Algorithm~2.4 is bounded by
\begin{equation}\label{CRTupdate}
O\bigl(nh\M(\log P)+h\M(\log M+n\log n)\bigr).
\end{equation}
We have $\log M \le n\log \pmax = O(|D|^{1/2}\log|D|\llog|D|)$, by bounds (iii) and (vi), thus (\ref{CRTupdate}) is bounded by $O(|D|\log^{3+\epsilon}|D|)$, using bound (i) on $h$.  The cost of Algorithm~2.5 in Step 5 is $O(h\M(\log P))=O(|D|^{1/2+\epsilon})$, with $\log P=O(\log^3|D|)$.  The space required is $O(h(\log|D|+\log P))$, which matches the bound in the theorem.

For $\log P > \mu \log^3|D|$, we apply the hybrid approach of Section \ref{subsection:LargeP}, whose costs are bounded in (\ref{equation:CRThybrid}).  Using the bounds on $\log M$, $n$, and $h$, we again obtain an $O(|D|\log^{3+\epsilon}|D|)$ time for all CRT computations, and the space is as above.
\end{proof}

The CRT approach is particularly well suited to a distributed implementation; one simply partitions the primes in $S$ among the available processors.  The precomputation steps in Algorithm~2 have complexity $O(|D|^{1/2+\epsilon})$, under the GRH, and this is comparable to the complexity of Algorithm~1.  Parallelism can be applied here, but in practice we are happy to repeat the precomputation on each processor.

When $\log P$ is polynomially bounded in $\log|D|$, the postcomputation can be performed in time $O(|D|^{1/2+\epsilon})$ by aggregating the CRT sums, with the final result $H_D\bmod P$ available on a single node.   When $P$ is larger, as when computing $H_D$ over $\Z$, we may instead have each processor handle the postcomputation for a subset of the coefficients of $H_D$, leaving the final result distributed among the processors.

We do not attempt a detailed analysis of the parallel complexity here, but note the following corollary, which follows from the discussion above.

\begin{corollary}[\textbf{GRH}]\label{cor:parallel}
There is a parallel algorithm to compute $H_D \bmod P$ on $O(|D|^{1/2+\epsilon})$ processors that uses $O(|D|^{1/2+\epsilon})$ time and space per processor.
\end{corollary}

\subsection{A heuristic analysis}\label{subsection:heuristics}
To obtain complexity estimates that better predict the actual performance of Algorithms 1 and 2, we consider a na\"{i}ve probabilistic model.  We assume that each positive integer $m$ is prime with probability $1/\log m$, and that for each prime $\ell\nmid D$ we have $\inkron{D}{\ell}=1$ with probability 1/2.  For a prime $\ell$ with $\inkron{D}{\ell}=1$ we further assume that if $\alpha,\alpha^{-1}\in\cl(D)$ are distinct classes containing an ideal of norm $\ell$, then $\alpha$ corresponds to a random element of $\clD$ uniformly distributed among the elements of order greater than 2.  Most critically, we suppose that all these probabilities are independent.  This last assumption is obviously false, but when applied on a large scale this model yields empirically accurate predictions.

Compared to the GRH-based analysis, these assumptions do not change the space complexity, nor bounds (i)--(iii), but significantly improve bounds (iv)--(vii).

\smallskip 

\noindent
(\textbf{H}) Our heuristic model predicts the following:
\begin{enumerate}
\setcounter{enumi}{3}
\item
$\boldsymbol{\lmax} = O(\log^{1+\epsilon}|D|).$
\vspace{4pt}
\item 
$\boldsymbol{z} = O(|D|^{1/2}\log^{1/2+\epsilon}|D|).$
\vspace{4pt}
\item 
$\boldsymbol{\pmax} = O(|D|\log^{1+\epsilon}|D|).$
\vspace{4pt}
\item 
$\boldsymbol{\vmax} = O(\log^{1/2+\epsilon}|D|).$
\end{enumerate}

Applying these to the analysis of Section \ref{section:Complexity} yields an $O(|D|\log^{3+\epsilon}|D|)$ bound on the expected running time of Algorithm~2, matching the heuristic result in \cite{Belding:HilbertClassPolynomial}.

It is claimed in \cite[\S 5.4]{Belding:HilbertClassPolynomial} that applying the bounds (i) and (ii) to \cite[Thm.~1.1]{Enge:FloatingPoint} also yields a heuristic complexity of $O(|D|\log^{3+\epsilon}|D|)$ when using the floating-point method to compute $H_D$. This is incorrect, the implied bound is actually $O(|D|\log^{4+\epsilon}|D|)$ (as confirmed by the author of \cite{Enge:FloatingPoint}).
\smallskip

One may reasonably question how accurate our $O(|D|\log^{3+\epsilon}|D|)$ estimate is in practice, since it assumes the Fast Fourier Transform (FFT) is used for all multiplications.  The cost $\M(n)$ arises in three distinct contexts:
\begin{enumerate}
\item[(a)]
The cost of operations in $\Fp$ is bounded by $O(\M(\log p))$.
\item[(b)]
Finding a root of $\Phi_\ell(X,j_i)/(X-j_{i-1})$ uses $O(\M(\ell)\log p)$ $\Fp$-operations.
\item[(c)]
Computing $\prod(X-j)$ uses $O(\M(h)\log h)$ $\Fp$-operations.
\end{enumerate}
In case (a) we actually expect $\lg \pmax$ to be smaller than the word size of our CPU, so multiplications in $\Fp$ effectively have unit cost.  For (b), $\ell$ is typically in the range where either schoolbook or Karatsuba-based multiplication should be used.  It is only in case (c) that FFT-based algorithms may be profitably applied.

In order to better estimate the running time of Algorithm 1 (which effectively determines the running time of Algorithm 2) we break out the cost of each step, expressing all bounds in terms of $\Fp$-operations.

\begin{center}
\begin{tabular}{ll}
Step & Complexity ($\Fp$-operations)\\\hline
1. Find $j\in\Ellt(\Fp)$ & $O(|D|^{1/2}\log^{3/2+\epsilon}|D|)$\\
2. Find $j'\in\EllD(\Fp)$ & negligible\\
3. Enumerate $\EllD(\Fp)$ & $O(|D|^{1/2}\log^{1+\omega+\epsilon}|D|)$\\
4. Compute $\prod_{j\in\EllD(\Fp)} (X-j)$ & $O(|D|^{1/2}\log^{2+\epsilon}|D|)$\\\hline
\end{tabular}
\vspace{6pt}\\
\textsc{Table} 1. (\textbf{H}) Heuristic complexity of Algorithm 1.
\vspace{4pt}
\end{center}

The value of $\omega$ depends on our estimate for $\M(\ell)$.  One can find values of $D$ in the feasible range where $\lmax$ is over 300, see \cite{Jacobson:NumericalResults,Jacobson:SupplementaryNumericalResults}, and here it is reasonable to assume $\M(\ell)=\ell^\omega$ with $\omega=\lg 3 \approx 1.585$.  In the worst case, Step 3 dominates.

However, the critical parameter is $\ell_1$, the least cost $\ell_i$ used by Algorithm 1.3.  If $\ell_1\nmid D$ we expect it to be used in the overwhelming majority of the steps taken by Algorithm 1.3.  As with $\lmax$, it is possible to find feasible $D$ for which $\ell_1$ is fairly large (over 100), but such cases are extremely rare.  If we average over $D$ in some large interval, our heuristic model predicts $\ell_1=O(1)$ (in fact ${\rm\bf E}[\ell_1] < 4$).  We typically have $\M(\ell_1)=O(1)$ and use $\omega=0$.  In almost all cases, Step 4 dominates.

The relative cost of Step 4 is not significant for small $|D|$, due to the excellent constant factors in the algorithms available for polynomial multiplication, but its asymptotic behavior becomes evident as $|D|$ grows (see Tables 3 and 4).

\section{Computational Results}\label{section:Performance}

To assess the performance of the new algorithm in a practical application, we used it to construct pairing-friendly curves suitable for cryptographic use, a task that often requires large discriminants.  We constructed ordinary elliptic curves of prime order and embedding degree $k$ over a prime field $\Fq$ such that either
$$\text{$k=6$ and $170< \lg q < 192$,\qquad or\qquad $k=10$ and $220< \lg q < 256$.}$$
These parameters were chosen using the guidelines in \cite{Freeman:PairingTaxonomy}, and have particularly desirable performance and security characteristics.  For additional background on pairing-based cryptography we refer to \cite[Ch.~24]{Cohen:HECHECC}.

To obtain suitable discriminants we used algorithms in \cite{Karabina:MNTCurves} (for $k=6$) and \cite{Freeman:EmbeddingDegree10} (for $k=10$) that were optimized to search for $q$ within a specified range.
This produced a set $\DPF$ of nearly 2000 fundamental discriminants (1722 with $k=6$ and 254 with $k=10$), with $|D|$ ranging from about $10^7$ to just over $10^{13}$ (almost all greater than $10^{10}$).  We selected 200 representative discriminants from $\DPF$ for our tests, including those that potentially posed the greatest difficulty, due to an unusually large value of $\ell_1$ or $h(D)$.

To each selected discriminant we applied the CM method, using Algorithm 2 to compute $H_D\bmod P$ (with $P=q$).  After finding a root $j$ of $H_D(X)$ over $\Fq$, we construct an elliptic curve $E$ with this $j$-invariant and ensure that the trace of $E$ has the correct sign.\footnote{One may apply the method of \cite{Rubin:CMTwist}, or simply compute $NQ$ for a nonzero point $Q\in E(\Fq)$, where $N$ is the desired (prime) order of $E(\Fq)$, and switch to a quadratic twist of $E$ if $NQ\ne 0$.}

\subsection{Implementation}
The algorithms described in this paper were implemented using the GNU C/C++ compiler \cite{GNU} and the GMP library \cite{GMP} on a 64-bit Linux platform.  Multiplication of large polynomials was handled by the zn\_poly library developed by Harvey \cite{Harvey:zn_poly}, based on the algorithm in \cite{Harvey:KroneckerSubstitution}.

The hardware platform included sixteen 2.8 GHz AMD Athlon processors, each with two cores.  Up to 32 cores were used in each test (with essentially linear speedup), but for consistency we report total cpu times, not elapsed times.  Memory utilization figures are per core, and can be achieved using a single core.

\subsection{Distribution of test discriminants}
To construct a curve of odd order over a field of odd characteristic we must have $D\equiv 5\bmod 8$, and this necessarily applies to $D\in\DPF$.  
We then have $\inkron{D}{2}=-1$, which implies $\ell_1\ge 3$, and also tends to make $h(D)$ smaller than it would be for an arbitrary discriminant.  Averaging over all discriminants up to an asymptotically large bound, we expect
$$L(1,\chi_D)=\frac{\pi h(D)}{\sqrt{|D|}}\quad\longrightarrow\quad C\pi^2/6 \approx 1.45,$$
where $C=\prod_p\bigl(1-1/(p^2(p+1))\bigr)$, see \cite[p.~296]{Cohen:CANT} (and see \cite{Jacobson:NumericalResults} for actual data).
Among the 1722 discriminants we found for $k=6$, the average value of $L(1,\chi_D)$ is about $0.55$, close to the typical value for $D\equiv 5\bmod 8$.  For $k=10$ we have the further constraint $\ell_1\ge 7$, and the average value of $L(1,\chi_D)$ is approximately $0.40$.

While we regard the discriminants in $\DPF$ as representative for the application considered, in order to assess the performance of Algorithm 2 in more extreme cases we also conducted tests using discriminants with very large values of $L(1,\chi_D)$.  These results are presented in Section \ref{subsection:bigh}.

\begin{table}
\begin{center}
\begin{tabular}{@{}lrrr@{}}
&Example 1& Example 2& Example 3\\\midrule
$|D|$&$13,569,850,003$&$11,039,933,587$&$12,901,800,539$\\
$h(D)$&20,203&11,280&54,706\\
$L(1,\chi_D)$&0.54&0.34&1.51\\
$b$&2,272,566&1,359,136&5,469,778\\
$n$&63,682&39,640&142,874\\
$z$&755,637&734,040&905,892\\
$\lceil\lg\pmax\rceil$&40&38&43\\
$\vmax$&12&8&32\\
$(\ell_1^{r_1},\ldots,\ell_k^{r_k})$&$(7^{20203})$&$(17^{1128},19^{10})$&$(3^{27038},5^2)$\\
\midrule
Step 1&0.0s&0.0s&0.0s\\
Step 2&1.2s&0.5s&4.0s\\
Step 3&0.6s&0.3s&2.0s\\
Step 4&23,300s&26,000s&61,000s\\
Step 5&0.0s&0.0s&0.0s\\
$(T_{\rm f},T_{\rm e},T_{\rm b})$&(57,32,11)&(51,47,2)&(53,20,27)\\
\midrule
throughput&2.0Mb/s&0.6Mb/s&4.9Mb/s\\
memory&3.9MB&2.1MB&9.4MB\\
total data&5.7GB&1.9GB&37GB\\
\midrule
Solve $H_D(X)=0$ over $\Fq$&127s&86s&332s\\
\bottomrule
\end{tabular}
\\
\vspace{12pt}
\textsc{Table} 2. Example computations.\\
\vspace{2pt}
\footnotesize
(2.8 GHz AMD Athlon)
\normalsize
\end{center}
\end{table}

\subsection{Examples}
Table 2 summarizes computations for three discriminants of comparable size, with $|D|\approx 10^{10}$.  These represent a typical case (Example 1) and two ``worst" cases (Examples 2 and 3).  The parameters appearing in the top section of the table are as defined in Section \ref{section:Complexity}.  The next section of the table contains timings for each step of Algorithm 2.

As predicted by the asymptotic analysis, essentially all of the time is spent in Step 4, which calls Algorithm~1 for each prime $p\in S$.
There are three principal components in the running time of Algorithm 1:

\begin{center}
\begin{tabular}{ll}
$T_{\rm f}$: &time spent in Step 1 finding a curve in $\Ellt(\Fp)$;\\
$T_{\rm e}$: &time spent in Step 3 enumerating $\EllD(\Fp)$;\\
$T_{\rm b}$:&time spent in Step 4 building $H_D(X)=\prod_{j\in\EllD(\Fp)}(X-j)\bmod p$.
\end{tabular}
\end{center}

These are listed in Table 2 as percentages of the total time $T$.  The time spent elsewhere ($T-T_{\rm f}-T_{\rm e}-T_{\rm b})$ is well under 1\% of $T$.

The third section in Table 2 lists the throughput, memory utilization, and total data processed during the computation.\footnote{The suffixes Mb, MB, and GB indicate $10^6$ bits, $10^6$ bytes and $10^9$ bytes, respectively.}  The total data is defined as the product of the number of coefficients $h(D)$ and the height bound $b$.  This approximates the total size of $H_D$, typically overestimating it by about 10\% (the actual sizes of $H_D$ for the three examples are 5.3GB, 1.8GB and 34GB respectively).  The throughput is then the total data divided by the total time.  Memory figures include all working storage and overhead due to data alignment (to word boundaries and to powers of 2 in FFT computations), but exclude fixed operating system overhead of about 4MB.  
The last row of the table lists the time to find a root of $H_D$ over $\Fq$, although this task is not actually performed by Algorithm 2.
\medskip

Example 1 represents a typical case: $L(1,\chi_D)$ is close to the mean of $0.55$, and $\ell_1=7$ is just above the median of $5$ (over $D\in\DPF$). Example 2 has an unusually large $\ell_1=17$ (exceeded by fewer than 1\% of $D\in\DPF$), while  Example 3 has an unusually large $L(1,\chi_D)\approx 1.51$ (exceeded by fewer than 1\% of $D\in\DPF$).

In Example 2, the large $\ell_1$ increases $T_{\rm e}$ substantially, despite the smaller $h(D)$.  The smaller $L(1,\chi_D)$ tends to increase the running time of individual calls to Algorithm 1.1, but at the same time $n$ decreases so that overall $T_{\rm f}$ decreases slightly.  The smaller values of $h(D)$ and $n$ both serve to decrease $T_{\rm b}$ significantly.

In Example 3 the large $L(1,\chi_D)$ decreases the cost of individual calls to Algorithm 1.1, but increases $n$ substantially so that overall $T_{\rm f}$ increases.  However, $T_{\rm e}$ and $T_{\rm b}$ increase even more, especially $T_{\rm b}$.  Despite the longer running time, this scenario results in the highest throughput of the three examples.

\subsection{Scaling}
Table 3 summarizes the performance of Algorithm 2 for $D\in \DPF$ ranging over six orders of magnitude.  We selected examples whose performance was near the median value for discriminants of comparable size.  We note the quasi-linear growth of $T$, and the increasing value $T_{\rm b}$ as a percentage of $T$, consistent with our heuristic prediction that this component is asymptotically dominant.
 
\begin{table}
\begin{center}
\begin{tabular}{@{}rrrrrrr@{}}
$|D|$ &$h(D)$&cpu secs&$(T_{\rm f},T_{\rm e},T_{\rm b})$&Mb/s&memory&data\\
\midrule
$116,799,691$&$2,112$&156&(65,28,7)&2.6&0.5MB&52MB\\
$1,218,951,379$&$6,320$&1,650&(64,26,8)&2.5&1.1MB&520MB\\
$13,569,850,003$&$20,203$&23,400&(57,33,10)&2.0&3.9MB&5.7GB\\
$126,930,891,691$&$56,282$&195,000&(66,22,12)&2.0&9.5MB&50GB\\
$1,009,088,517,019$&$181,584$&2,160,000&(64,20,16)&2.0&34MB&535GB\\
$10,028,144,961,139$&$521,304$&20,600,000&(63,20,17)&1.9&84MB&5.0TB\\
\bottomrule
\end{tabular}
\\
\vspace{8pt}
\textsc{Table} 3.  Performance for typical $D\in \DPF$.\\
\vspace{2pt}
\footnotesize
(2.8 GHz AMD Athlon)
\normalsize
\end{center}
\end{table}
Up to 32 cores were applied to the computations in Table 3.  In all but the smallest example we can effectively achieve a 32x speedup.  The actual elapsed time for the largest discriminant was about 8 days, while the second largest took less than a day.   
As suggested by Corollary \ref{cor:parallel}, these computations could be usefully distributed across many more processors.  The low memory requirements provide headroom for much larger computations: each of our cores had 2GB of memory, but less than 100MB was used.

Below is an example of a curve constructed using $D=-10,028,144,961,139$, the largest discriminant listed in Table 3.
The elliptic curve
$$y^2=x^3-3x+3338561401570133202017008597803337396411439360229378547$$
has embedding degree 6 over the finite field $\Fq$ with
$$q=30518311673028635209000068713843412774183984182022701057.$$
This curve has prime order $N=q+1-t$, where
$$t=5524338120809463560527395583.$$
There are a total of $h(D)=521,304$ nonisomorphic curves with the same order that may be constructed using $H_D \bmod q$.
A complete list of curves for all the discriminants tested is available at \url{http://math.mit.edu/~drew}.

\subsection{Discriminants with large $\boldsymbol{L(1,\chi_D)}$}\label{subsection:bigh}

Table 4 shows the performance of Algorithm~2 on discriminants specifically chosen to make $L(1,\chi_D)$ extremely large, between 6.8 and 7.8.  These discriminants are not in $\DPF$, and are likely the smallest possible for the class numbers listed (but we do not guarantee this).  In each case we computed $H_D$ modulo a 256-bit prime $P$.  The timings would not change significantly for larger $P$, but the space would increase.

The first discriminant $D=-2,093,236,031$ in Table~4 also appears in Table~1 of \cite{Enge:FloatingPoint}.  Scaled to the same processor speed, Algorithm 2 computes $H_D\bmod P$ using less than half the cpu time spent by the floating-point approximation method to compute a class polynomial over $\Z$ for the same $D$ (this polynomial would then need to be reduced mod $P$ in order to apply the CM method).
Most significantly, the memory required is about 20 MB versus 5 GB.

This comparison is remarkable, given that the height bound $b=7,338,789$ for $H_D$ is nearly 28 times larger than the 264,727 bits of precision used in \cite{Enge:FloatingPoint}, where the class polynomial for the double-eta quotient $\mathfrak{w}_{3,13}$ was computed instead of the Hilbert class polynomial.  The difference in throughput is thus much greater than the difference in running times: 7.5 Mb/s versus 0.10 Mb/s.
 
\begin{table}
\begin{center}
\begin{tabular}{@{}rrrrrrr@{}}
$|D|$ &$h(D)$&cpu secs&$(T_{\rm f},T_{\rm e},T_{\rm b})$&Mb/s&memory&data\\
\midrule
$2,093,236,031$&100,000&98,800&(25,19,56)&7.5&18MB&93GB\\
$8,364,609,959$&200,000&472,000&(24,17,59)&6.8&36MB&400GB\\
$17,131,564,271$&300,000&1,240,000&(20,15,65)&5.9&61MB&920GB\\
$30,541,342,079$&400,000&2,090,000&(21,16,63)&6.4&71MB&1.7TB\\
$42,905,564,831$&500,000&3,050,000&(22,17,61)&6.9&81MB&2.6TB\\
$67,034,296,559$&600,000&5,630,000&(18,14,68)&5.6&121MB&3.9TB\\
$82,961,887,511$&700,000&7,180,000&(19,14,67)&5.9&132MB&5.3TB\\
$113,625,590,399$&800,000&9,520,000&(19,15,66)&5.9&142MB&7.1TB\\
$133,465,791,359$&900,000&11,500,000&(20,15,65)&6.2&152MB&9.0TB\\
$170,868,609,071$&1,000,000&14,200,000&(20,16,64)&6.3&163MB&11.2TB\\
\bottomrule
\end{tabular}
\\
\vspace{4pt}
\textsc{Table} 4.  Performance when $L(1,\chi_D)$ is large.\\
\vspace{2pt}
\scriptsize
(2.8 GHz AMD Athlon)
\normalsize
\end{center}
\end{table}

\section{Acknowledgments}
The author thanks Daniel J. Bernstein for initially suggesting this project, and also Gaetan Bisson, Reinier Br\"{o}ker, Andreas Enge, David Harvey, Tanja Lange, and Kristin Lauter for their support and feedback on an early draft of this paper.  I am also grateful to the referee for careful reading and many helpful suggestions.

\section*{Appendix 1}
This appendix proves Lemma~\ref{lemma:Bbound}, which bounds the coefficients of the Hilbert class polynomial $H_D(X)$, and Lemma~\ref{lemma:vbound} which bounds the Hurwitz class number $H(-v^2D)$ in terms of $v$ and $H(-D)$.

Let $B$ denote an upper bound on the absolute values of the coefficients of $H_D(X)$.  In the literature one finds many values for $B$ (or $\log B$), but due to an unfortunate series of typographical errors, most are either incorrect \cite[p.~285]{Belding:HilbertClassPolynomial}, exponentially larger than necessary (\cite[Eq.~22]{Atkin:ECPP} and \cite[p.~151]{Blake:EllipticCurves}), or heuristics that do hold for all $D$ (\cite[Eq.~3.1]{Agashe:CRTClassPolynomial} and \cite[p.~2431]{Broker:pAdicClassPolynomial}).  In \cite[Thm.~1.2]{Enge:FloatingPoint}, Enge gives a rigorous and fully explicit value for $B$ that is empirically accurate to within a constant factor, but still larger than desirable for practical application.  Provided one is prepared to enumerate the elements of $\cl(D)$, a much tighter bound is given by the lemma below, whose proof is derived directly from Enge's analysis in \cite[\S 4]{Enge:FloatingPoint}.

\begin{lemma}\label{lemma:Bbound}
For a quadratic discriminant $D< 0$, let $(a_1,b_1,c_1),\ldots,(a_h,b_h,c_h)$ be the sequence of reduced, primitive binary quadratic forms of discriminant $D$ with $0<a_1\le \cdots\le a_h$, where $h=h(D)$.  Let $M_k=\exp(\pi\sqrt{|D|}/a_k)+C$, where $C=2114.567$.  Then the coefficients of $H_D(X)$ have absolute values bounded by
$$
B = \binom{h}{m} M_h^{-m} \prod_{k=1}^h M_k,
$$
where $m = \left\lfloor\frac{h+1}{M_h+1}\right\rfloor$.  We also have $\log B = O(|D|^{1/2}\log^2|D|)$, and under the GRH, $\log B = O(|D|^{1/2}\log|D|\llog|D|)$.
\end{lemma}
\begin{proof}
We may write $H_D$ as
$$
H_D(X) = \prod_{k=0}^h\bigl(X-j(\tau_k)\bigr),
$$
where $\tau_k=(-b_k+\sqrt{D})/2a_k$.  With $q_k=e^{2\pi i \tau_k}$, we have $M_k=|1/q_k|+C$, where the constant $C$ bounds $|j(\tau_k)-1/q_k|$, as shown in \cite[p.~1094]{Enge:FloatingPoint}.  Thus $|j(\tau_k)|\le M_k$, and the absolute value of the coefficient of $X^n$ in $H_D(X)$ is bounded by
\begin{equation}\label{equation:Xnbound}
B_n = \binom{h}{n}\prod_{k=1}^{h-n}M_k.
\end{equation}
We now argue that $B_n\le B$.  For $n > m$ we have $n > (h+1)/(M_h+1)$ and
$$
\binom{h}{n}\Bigm/\binom{h}{n-1} = \frac{h-n+1}{n} < M_h.
$$
This implies $B_n < B_{n-1}$.  For $0< n \le m$ we have $\binom{h}{n}/\binom{h}{n-1} \ge M_h$, which implies
$$
B_0 \le B_1M_h/M_h \le B_2M_{h-1}M_h/M_h^2\le \cdots \le B_mM_{h-m+1}\cdots M_h/M_h^m = B.
$$
It follows that $B$ bounds every $B_n$.

The bound $\log B = O(|D|^{1/2}\log^2|D|)$ follows from $h=O(|D|^{1/2}\log|D|)$, as proven in \cite{Schur:ClassNumberBound}, and the bound $\sum_k\frac{1}{a_k}=O(\log^2|D|)$, proven in \cite[Lemma~2.2]{Schoof:Exponents}.  As shown in \cite[Lemma~2]{Belding:HilbertClassPolynomial}, under the GRH the bound $\sum_k\frac{1}{a_k}=O(\log|D|\llog|D|)$ follows from \cite{Littlewood:ClassNumber}, which yields $\log B = O(|D|^{1/2}\log|D|\llog|D|)$.
\end{proof}

In practice the bound given by Lemma~\ref{lemma:Bbound} is close to, and often better than, the heuristic bound
$B = \binom{h}{\lfloor{h/2\rfloor}}\exp(\pi\sqrt{|D|}\sum_k\frac{1}{a_k})$
that is sometimes used, even though the latter bound is not actually valid for all $D$ (such as $D=-99$).

\begin{lemma}\label{lemma:vbound}
Let $D$ be a negative discriminant, let $v\ge 2$ be an integer, and let $x$ be the largest prime for which $\prod_{p\le x}p\le v$, where $p$ ranges over primes.  Let $H(n)$ denote the Hurwitz class number.  The following bounds hold:
$$1\sle \frac{H(-v^2D)}{vH(-D)}\sle\prod_{p\le x}\frac{p+1}{p-1} \slt 11\llog^2(v+4).$$
\end{lemma}
\begin{proof}
Let $u$ be the conductor of $D$, so that $D=u^2D_0$.  Then
\begin{equation}\label{equation:vb1}
H(-v^2D)=H(-(uv)^2D_0)=\sum_{d|uv}\frac{2h(d^2D_0)}{w(d^2D_0)},
\end{equation}
where $w(d^2D_0)=|\O_{d^2D_0}^*|$ is 2, 4, or 6 \cite[Lemma 5.3.7]{Cohen:CANT}.  We also have \cite[p.~233]{Cohen:CANT}
$$\frac{h(d^2D_0)}{w(d^2D_0)}=\frac{h(D_0)}{w(D_0)}d\prod_{p|d}\left(1-\frac{\chi_p}{p}\right),$$
where $\chi_p=\left(\frac{D_0}{p}\right)$ is $-1$, 0, or 1.  Regarding $D_0$ as fixed, we note that
$$H(-n^2D_0)=\frac{2h(D_0)}{w(D_0)}\sum_{d|n}d\prod_{p|d}\left(1-\frac{\chi_p}{p}\right)$$
is a multiplicative function of $n$, which yields
\begin{equation*}
H(-n^2D_0)=\frac{2h(D_0)}{w(D_0)}\prod_p\left(1+\left(p^{\nu_p(n)}-1\right)(p-\chi_p)/(p-1)\right).
\end{equation*}
where $\nu_p(n)$ is the $p$-adic valuation.  From (\ref{equation:vb1}) we obtain
\begin{equation}\label{equation:vb2}
\frac{H(-v^2D)}{vH(-D)}=\frac{\prod_p\left(1+\left(p^{\nu_p(u)+\nu_p(v)}-1\right)(p-\chi_p)/(p-1)\right)}
{v\prod_p\left(1+\left(p^{\nu_p(u)}-1\right)(p-\chi_p)/(p-1)\right)}.
\end{equation}
Fixing $D=u^2D_0$, we regard (\ref{equation:vb2}) as a multiplicative function of $v$.  For $v=p^k$:
\begin{equation*}
\frac{H(-p^2kD)}{p^kH(-D)}=\frac{\left(1+\left(p^{\nu_p(u)+k}-1\right)(p-\chi_p)/(p-1)\right)}
{p^k\left(1+\left(p^{\nu_p(u)}-1\right)(p-\chi_p)/(p-1)\right)}.
\end{equation*}
This value is minimized when $\chi_p=1$, in which case it is 1, yielding the first inequality in the lemma.
It is maximized when $\chi_p=-1$, in which case one finds
$$\frac{\left(1+\left(p^{\nu_p(u)+k}-1\right)(p+1)/(p-1)\right)}
{p^k\left(1+\left(p^{\nu_p(u)}-1\right)(p+1)/(p-1)\right)}\le \frac{p+1}{p-1},$$
for all nonnegative integers $k$ and $\nu_p(u)$.  We thus obtain from (\ref{equation:vb2})
\begin{equation*}
\frac{H(-v^2D)}{vH(-D)}\le\prod_{p|v}\frac{p+1}{p-1}\le\prod_{p\le x}\frac{p+1}{p-1},
\end{equation*}
proving the second inequality in the lemma.  To prove the third inequality, we first note
that for $v\ge\prod_{p\le x} p$ the inequality holds for each prime $x < 41$, by a machine calculation, so we assume $x\ge 41$.
We then have
$$
\log\prod_{p\le x}\frac{p+1}{p-1}=\sum_{p\le x}\log\left(1+\frac{2}{p-1}\right)\le\sum_{p\le x}\frac{2}{p-1}=
2\sum_{p\le x}\frac{1}{p} + 2\sum_{p\le x}\frac{1}{p(p-1)}.
$$
We now apply the bound $\sum_{p\le x}\frac{1}{p}<\llog x + B_1 + 1/(\log x)^2$ from \cite[3.20]{Rosser:PrimeSums}, where $B_1= 0.261497\ldots$, and also $\sum_p\frac{1}{p(p-1)} = 0.773156\ldots$ from \cite{Cohen:HardyLittlewood}, to obtain
$$
\log\prod_{p\le x}\frac{p+1}{p-1} < 2\llog x + 2.218,
$$
valid for $x\ge 41$.  This yields $\prod_{p\le x}\frac{p+1}{p-1} < 9.189\cdot \log^2 x$.  We also have the bound $x(1-1/\log x) < \sum_{p\le x}\log p$, valid for $x\ge 41$, by \cite[3.16]{Rosser:PrimeSums}, which implies
$$
\prod_{p\le x}\frac{p+1}{p-1} < 9.189\cdot\log^2(1.369\cdot \log v),
$$
For $x\ge 41$ we have $\log v > 30$, and the RHS is then smaller than $11\llog^2(v+4)$.
\end{proof}

\section*{Appendix 2}\label{section:Appendix2}
Here we list some of the torsion constraints used to accelerate the search for an elliptic curve $E/\Fp$ with $p+1\pm t$ points, as described in Section 3.
Each constraint has the form $m = a\cdot b\cdot N$, where $a$ is a power of 2 and $b$ is a power of 3.
Curves with a point of order $N$ are generated using a plane model for $X_1(N)$ as in \cite{Sutherland:PrescribedTorsion}, then filtered to ensure that the constraints implied by $a$ and $b$ are also met.  When $a$ or $b$ is expressed in exponential notation, it is meant to control the exact power of 2 or 3 that divides $\#E$.  The torsion constraint $14=2^0\cdot3^0\cdot 14$, for example, indicates that $\#E$ is divisible by 14 but not divisible by 3 or 4.

Efficient methods for analyzing the Sylow 2-subgroup of $E(\Fp)$ are considered in \cite{Miret:EC2Sylow,Sutherland:PrescribedTorsion}, and for 3-torsion we use the 3-division polynomial \cite[\S~3.2]{Washington:EllipticCurves}.
For the sake of brevity, here we consider constraints on the Sylow 2-subgroup only up to 4-torsion, but one may obtain minor improvements using $2^k$-torsion for larger $k$.

The benefit of each constraint is computed as 1/$r$, where $r$ is the proportion of elliptic curves $E/\Fp$ that satisfy the constraint.  We derive $r$ using \cite[Thm.~1.1]{Howe:EllipticCurveOrders}, under the simplifying assumption that if $N$ divides $\#E$, then $E(\Fp)$ contains a point of order $N$ (necessarily true when the square part of $N$ is coprime to $p-1$).  A more precise estimate may be obtained from \cite[Thm.~3.15]{Gekeler:ECGroupStructures}.  Table~5 assumes that $p\equiv 1\bmod 3$ and $p\not\equiv 1\bmod\ell$ for primes $\ell>3$ dividing $N$.  It is easily adjusted to other cases via \cite[Thm.~1.1]{Howe:EllipticCurveOrders}; this will change the rankings only slightly.

The cost of each constraint was determined empirically (and is somewhat implementation dependent).
For a random set of primes $p$ of suitable size (30-50 bits) we measured the average time to: (1) generate a curve $E/\Fp$ satisfying the constraint, (2) obtain a random point $P\in E(\Fp)$, and (3) compute the points $(p+1)P$ and $tP$.  This is compared to the cost of (2) and (3) alone (the ``null case"
 for Algorithm~1.1, excluding \textsc{TestCurveOrder} which is rarely called).
The parametrizations of \cite{Atkin:ECMCurves} combine (1) and (2), enabling a cost of less than 1.0 in some cases.

\pagebreak
The rankings in Table~5 assume each constraint is applicable to both $N_0=p+1-t$ and $N_1=p+1+t$; if not, the effective ratio is about half the listed value ($9/16$, on average).  For given values of $p$ and $t$, we thus consider three possible constraints, one satisfied by $N_0$, one by $N_1$, and one by both, and pick the best of the three.

\begin{table}
\begin{center}
\begin{tabular}{@{}rrrrrrrrrr@{}}
$m$ &torsion&benefit&cost&ratio&$\qquad m$&torsion&benefit&cost&ratio\\
\midrule
  33 & $2^0 \cdot 3 \cdot 11$ & 80.0 &  2.3 & 34.3 &   44 & $4\cdot 11 $ & 24.0 &  1.8 & 13.0\\
  39 & $2^0 \cdot 3 \cdot 13$ & 96.0 &  3.0 & 31.5 &   93 & $2^0 \cdot 3 \cdot 31$ & 240.0 & 18.9 & 12.7\\
  51 & $2^0 \cdot 3 \cdot 17$ & 128.0 &  4.4 & 29.0 &   34 & $2^1 \cdot 17$ & 64.0 &  5.0 & 12.7\\
  15 & $2^0 \cdot 15$ & 32.0 &  1.2 & 26.4 &   28 & $2\cdot 14 $ & 14.4 &  1.2 & 12.4\\
  11 & $2^0 \cdot 11$ & 30.0 &  1.2 & 25.9 &   52 & $4\cdot 13 $ & 28.8 &  2.4 & 12.2\\
  57 & $2^0 \cdot 3 \cdot 19$ & 144.0 &  5.7 & 25.4 &   18 & $2^0 \cdot 18$ & 26.2 &  2.2 & 12.0\\
  66 & $2^1 \cdot 3 \cdot 11$ & 106.7 &  4.3 & 24.7 &   36 & $2\cdot 18 $ & 15.7 &  1.3 & 12.0\\
  21 & $2^0 \cdot 3 \cdot  7$ & 48.0 &  2.1 & 23.1 &   68 & $4\cdot 17 $ & 38.4 &  3.4 & 11.2\\
  69 & $2^0 \cdot 3 \cdot 23$ & 176.0 &  7.8 & 22.4 &   38 & $2^1 \cdot 19$ & 72.0 &  6.7 & 10.8\\
  78 & $2^1 \cdot 3 \cdot 13$ & 128.0 &  5.8 & 22.0 &   10 & $2^0 \cdot 10$ & 16.0 &  1.5 & 10.7\\
  13 & $2^0 \cdot 13$ & 36.0 &  1.6 & 21.8 &  174 & $2^1 \cdot 3 \cdot 29$ & 298.7 & 28.6 & 10.4\\
   9 & $2^0 \cdot  9$ & 19.6 &  1.0 & 20.2 &   20 & $2\cdot 10 $ &  9.6 &  0.9 & 10.4\\
 102 & $2^1 \cdot 3 \cdot 17$ & 170.7 &  8.5 & 20.0 &  348 & $ 4 \cdot 3 \cdot 29$ & 179.2 & 18.0 & 9.9\\
  42 & $2^0 \cdot 3 \cdot 14$ & 64.0 &  3.2 & 20.0 &   76 & $4\cdot 19 $ & 43.2 &  4.4 & 9.9\\
   7 & $2^0 \cdot  7$ & 18.0 &  0.9 & 19.6 &   46 & $2^1 \cdot 23$ & 88.0 &  9.2 & 9.5\\
 132 & $ 4 \cdot 3 \cdot 11$ & 64.0 &  3.3 & 19.2 &   29 & $2^0 \cdot 29$ & 84.0 &  8.9 & 9.5\\
  17 & $2^0 \cdot 17$ & 48.0 &  2.5 & 19.0 &   48 & $3\cdot 16$ & 21.3 &  2.3 & 9.4\\
 156 & $ 4 \cdot 3 \cdot 13$ & 76.8 &  4.2 & 18.2 &    3 & $2^0 \cdot  3$ &  8.0 &  0.9 & 9.2\\
 204 & $ 4 \cdot 3 \cdot 17$ & 102.4 &  5.9 & 17.5 &   92 & $4\cdot 23 $ & 52.8 &  6.0 & 8.8\\
 114 & $2^1 \cdot 3 \cdot 19$ & 192.0 & 11.0 & 17.4 &   12 & $ 12$ &  6.4 &  0.7 & 8.8\\
  30 & $2^1 \cdot 15$ & 42.7 &  2.5 & 16.9 &  186 & $2^1 \cdot 3 \cdot 31$ & 320.0 & 36.9 & 8.7\\
  84 & $ 2 \cdot 3 \cdot 14$ & 38.4 &  2.3 & 16.5 &   31 & $2^0 \cdot 31$ & 90.0 & 11.5 & 7.8\\
  19 & $2^0 \cdot 19$ & 54.0 &  3.3 & 16.4 &    6 & $2^0 \cdot  6$ & 10.7 &  1.4 & 7.4\\
  22 & $2^1 \cdot 11$ & 40.0 &  2.5 & 16.2 &   16 & $ 16$ &  8.0 &  1.1 & 7.2\\
 228 & $ 4 \cdot 3 \cdot 19$ & 115.2 &  7.4 & 15.7 &   58 & $2^1 \cdot 29$ & 112.0 & 17.4 & 6.4\\
  87 & $2^0 \cdot 3 \cdot 29$ & 224.0 & 14.5 & 15.5 &  116 & $4\cdot 29 $ & 67.2 & 11.0 & 6.1\\
 138 & $2^1 \cdot 3 \cdot 23$ & 234.7 & 15.3 & 15.4 &    8 & $  8$ &  4.0 &  0.7 & 5.9\\
  26 & $2^1 \cdot 13$ & 48.0 &  3.3 & 14.4 &   62 & $2^1 \cdot 31$ & 120.0 & 23.0 & 5.2\\
  23 & $2^0 \cdot 23$ & 66.0 &  4.7 & 14.2 &  124 & $4\cdot 31 $ & 72.0 & 14.2 & 5.1\\
 276 & $ 4 \cdot 3 \cdot 23$ & 140.8 & 10.0 & 14.0 &    2 & $2^0 \cdot  2$ &  4.0 &  0.9 & 4.3\\
  14 & $2^0 \cdot 14$ & 24.0 &  1.7 & 13.8 &    4 & $  4$ &  2.4 &  0.6 & 3.8\\
  60 & $4\cdot 15 $ & 25.6 &  1.9 & 13.5 &    1 & $2^0 \cdot  1$ &  3.0 &  0.8 & 3.7\\
   5 & $2^0 \cdot  5$ & 12.0 &  0.9 & 13.0\\
\bottomrule
\end{tabular}
\\
\vspace{6pt}
\textsc{Table} 5. Ranking of $m$-torsion constraints (for $p\equiv 1 \bmod 3$).\\
\vspace{6pt}
\small
Dominated constraints are not listed, e.g. $3\cdot 4\cdot 31$ is always inferior to 12.
\end{center}
\normalsize
\end{table}
\normalsize

\bibliographystyle{amsplain}
\providecommand{\bysame}{\leavevmode\hbox to3em{\hrulefill}\thinspace}
\providecommand{\MR}{\relax\ifhmode\unskip\space\fi MR }
\providecommand{\MRhref}[2]{%
  \href{http://www.ams.org/mathscinet-getitem?mr=#1}{#2}
}
\providecommand{\href}[2]{#2}

\end{document}